\documentclass{amsproc}
\usepackage{amsmath, amssymb,graphicx}

\usepackage{tikz}

\theoremstyle{plain}
\newtheorem{theorem}{Theorem}[section]
\newtheorem*{theorem*}{Theorem}
\newtheorem{proposition}[theorem]{Proposition}
\newtheorem{corollary}[theorem]{Corollary}
\newtheorem{lemma}[theorem]{Lemma}

\theoremstyle{definition}
\newtheorem{definition}[theorem]{Definition}

\theoremstyle{remark}
\newtheorem*{remark}{Remark}

\numberwithin{equation}{section}

\DeclareMathOperator{\LCD}{LCD}

\def \R {\mathbb{R}}
\def \C {\mathbb{C}}

\def \Z {\mathbb{Z}}
\def \E {\mathbb{E}}
\def \P {\mathbb{P}}

\def \EE {\mathcal{E}}
\def \NN {\mathcal{N}}

\def \LL {\mathcal{L}}
\def \MM {\mathcal{M}}

\def \a {\alpha}
\def \b {\beta}
\def \g {\gamma}
\def \e {\varepsilon}

\def \d {\delta}

\def \k {\kappa}
\def \l {\lambda}

\def \s {\sigma}

\def \z {\zeta}

\def \< {\langle}
\def \> {\rangle}
\def \^ {\widehat}

\def \dist {{\rm dist}}

\def \vol {{\rm vol}}

\def \supp {{\rm supp}}
\def \Sparse {{\mathit{Sparse}}}
\def \Comp {{\mathit{Comp}}}
\def \Incomp {{\mathit{Incomp}}}

\def \smin {{s_n}}

\newcommand{\pr}[2]{\langle {#1} , {#2} \rangle}
\newcommand{\norm}[1]{\left \| #1 \right \|}
\def \etc {,\ldots,}

\begin{document}
\title[Non-asymptotic theory]{Recent developments in non-asymptotic theory \\ of random matrices
   }

\author{Mark Rudelson
  }

\thanks{
  Partially supported by NSF DMS grant DMS 1161372.
}

\address{Department of Mathematics,
   University of Michigan,
   Ann Arbor, MI 48109, USA}
\email{rudelson@umich.edu}

\date{}


\begin{abstract}
Non-asymptotic theory of random matrices strives to investigate the spectral properties of random matrices, which are valid with high probability for matrices of a large fixed size. Results obtained in this framework find their applications in high-dimensional convexity, analysis of convergence of algorithms, as well as in random matrix theory itself. In these notes we survey some recent results in this area and describe the techniques aimed for obtaining explicit probability bounds.

\end{abstract}

\maketitle

\tableofcontents

\section{Introduction}

The classical random matrix theory is concerned with asymptotics of various spectral characteristics of families of random matrices, when the dimensions of the matrices tend to infinity. There are many examples when these characteristics, which are random variables themselves, converge to certain limit laws. This includes the celebrated Wigner semicircle law for the empirical measures of eigenvalues of random symmetric matrices, Marchenko--Pastur law, which is the limit of empirical measures of sample covariance matrices, Tracy--Widom distribution describing the limit of the first singular values of a sequence of random matrices, etc. \cite{AGZ}.
These limits are of paramount importance, yet in applications one usually needs information about the behavior of such characteristics for large, but fixed $n$. For instance in problems in convex geometry one constructs a random section of an $N$-dimensional convex body by taking the kernel or the range of a certain random matrix. Random matrices arise also in analysis of rates of convergence of computer science algorithms. In both cases, the dimension of the ambient space remains fixed, and one seeks explicit estimates of probabilities in terms of the dimension. For such problems knowing the limit behavior is of little help.

The problems involving estimates for a fixed finite dimension arise in the classical random matrix theory as well. One of the main approaches in deriving the limit laws is based on analysis of the Stieltjes transform of measures \cite{AGZ}. To derive the convergence of Stieltjes transforms, one frequently has to provide explicit bounds on the smallest singular value of a random matrix of a fixed size, which holds with high probability. This need arises, e.g., in derivation of the circular law \cite{GT, TV circ1, TV circ2} and the single ring theorem \cite{GKZ}.

 These questions led to development of non-asymptotic theory of random matrices, which provides probabilistic bounds for eigenvalues, singular values, etc. for random matrices of a large fixed size. The situation is roughly parallel to that arising for the sums of i.i.d. random variables, where the asymptotic and non-asymptotic results go hand in hand. The asymptotic behavior of the averages of $n$ i.i.d. random variables is governed by the Strong Law of Large Numbers establishing the almost sure convergence to the expectation. Yet, to assert that the average of a large number of random variables is close to the expectation, we need a non-asymptotic version, e.g. Hoeffding inequality. This inequality yields a subgaussian bound for the large deviations (see the details below). Such behavior suggests that the limit distribution of the deviation should be normal, which leads to an asymptotic result, the Central Limit Theorem (CLT). To use the CLT in evaluation of probabilities for random sums, we need its non-asymptotic version, namely the Berry--Esseen Theorem. This theorem provides in turn a crucial step in deriving another fundamental asymptotic result,  the Law of Iterated Logarithm.

  These notes discuss the methods of the non-asymptotic approach to the random matrix theory. We do not attempt to provide an exhaustive list of references (a reader can check the surveys \cite{DS}, \cite{RV ICM}, and \cite{V n-a}). Instead we concentrate on three essentially different examples, with the aim of presenting the methods and results in a maximally self-contained form. This approach inevitably leaves out several important recent developments, such as invertibility of random symmetric matrices \cite{V sym, Ng}, applications to the Circular Law \cite{GT, TV circ1, TV circ2}, and concentration for random determinants \cite{TV determinant Wigner, NV}. Yet, by restricting ourselves to a few results, we will be able to give a relatively complete picture of the ideas and methods involved in their proofs. We start with introduction to subgaussian random variables in Section \ref{sec: subgaussian}. In Sections \ref{sec: continuous}-\ref{sec: random normal} we obtain quantitative bounds for invertibility of random matrices with i.i.d. entries. As will be shown in Section \ref{sec: small ball}, the arithmetic structures play a crucial role here. Section \ref{sec: short Khinchin} studies a question arising in geometric functional analysis. Here the ambient space is Banach, and the approach combines the methods of the previous sections with the functional-analytic considerations. We will also touch upon majorising measures, which are a powerful tool for estimating suprema of random processes. Section \ref{sec: unitary} contains another quantitative invertibility result. Here we discuss a random unitary or orthogonal perturbation of a fixed matrix. Unlike in the first example, the arithmetic structure plays no role in this problem. The main difficulty is the dependence between the entries of a random matrix, and the method is based on the introduction of perturbations with independent entries.

  \subsection*{Acknowledgement} These notes are based in part on the material presented at
  the workshop
   ``Etats de la Recherche:
  Probability and geometry in interaction''
  at Paul Sabatier University in
  Toulouse,
    the mini-course given at the Warsaw University, and the Informal Analysis Seminar at Kent State University. The author is grateful to Franck Barthe, Michel Ledoux,
Rafal Latala, Krzystof Oleszkiewicz,  Michal Wojchehowski, Fedor Nazarov, Dmitry Ryabogin, and Artem Zvavitch for
their hospitality. The author is also grateful to Fedor Nazarov, Dmitry Ryabogin, and an anonimous referee for careful reading of the manuscript and many suggestions, which led to improvement of the presentation.

\section{Notation and basic definitions}

We shall consider random matrices of high order with independent
entries. For simplicity, we shall assume that the entries are
centered ($\E a_{j,k}=0$) and identically distributed (both
conditions may be relaxed).

Throughout these notes $\norm{\cdot}_p$ denotes the $\ell_p$ norm
\[
 \norm{x}_p=\left( \sum_{j=1}^n |x_j|^p \right)^{1/p}, \qquad 1 \le p<\infty,
\]
and $B_p^n$ stands for the unit ball of this norm.
The norm of an operator or a matrix will be denoted by $\norm{\cdot}$.
We use  $S^{n-1}$ for the unit Euclidean
sphere.
If $F$ is a finite set, then $|F|$ denotes the cardinality of $F$.
Letters $C, C', c$ etc. denote absolute constants.

If $N \ge n$ then an $N \times n$ matrix $A$ can be viewed as a
mapping of $\R^n$ into $\R^N$. Thus, a random matrix defines a
random $n$-dimensional section of $\R^N$. For geometric applications
we need to know that this matrix would not distort the metric too
much. Let us formulate it more precisely:
\begin{definition}
 Let $N \ge n$ and let $A$ be an $N \times n$ matrix. The condition
 number of the matrix $A$ is
 \[
   \k(A)= \frac{\max_{x \in S^{n-1}} \norm{Ax}_2}
               {\min_{x \in S^{n-1}} \norm{Ax}_2}.
 \]
 If $\min_{x \in S^{n-1}} \norm{Ax}_2=0$, we set $\k(A)= \infty$.
\end{definition}
The condition number of a matrix can be rewritten in terms of its
singular values.
\begin{definition}
 Let $N \ge n$ and let $A$ be an $N \times n$ matrix. The singular
 values of $A$ are the eigenvalues of $(A^*A)^{1/2}$, arranged in
 the decreasing order: $s_1(A) \ge s_2(A) \ge \ldots\ge s_n(A)$.
\end{definition}
The singular values of $A$ are the lengths of the semi-axes of the ellipsoid $A B_2^n$.
The first and the last singular values have a clear
functional-analytic meaning:
\[
  s_1(A)=\norm{A: \R^n \to \R^N},
\]
and
\[
  s_n(A)= \min_{x \in S^{n-1}} \norm{Ax}
  =1/\norm{A^{-1}: A \R^n \to \R^n},
\]
whenever $A$ has the full rank. In this notation $\k(A)=s_1(A)/s_n(A)$.

Therefore, to bound the condition number, we have to estimate the
first singular value from above, and the last one from below.
For matrices with i.i.d. random entries the first singular value is the most robust. It can be estimated using a simple $\e$-net argument, as will be shown in Proposition \ref{p: norm}. The last singular value presents a bigger challenge. We will obtain its bounds for ``tall'' rectangular matrices in Section \ref{sec: rectangular}, and for square matrices in Sections  \ref{sec: continuous}-\ref{sec: random normal}.

\section{Subgaussian random variables} \label{sec: subgaussian}

In this section we introduce an important class of random variables with strong tail decay properties.
 This class  contains the normal variables, as well as  all
bounded random variables.

\begin{definition}
Let $v>0$. A random variable  $\xi$ is called $v$-subgaussian if
 there exists a constant $C$  such that for any $t>0$
 \[
   \P(|\xi|>t) \le C e^{-vt^2}.
 \]
 A random variable $\xi$ is called centered if $\E \xi=0$.
\end{definition}
 If the parameter $v$ is an absolute constant, we call a $v$-subgaussian random variable subgaussian.
 We shall assume  that the random variable $\xi$ is
 non-degenerate, i.e. $\text{Var}(\xi)>0$.

The subgaussian condition can be formulated in a number of different
ways.
\begin{theorem}  \label{t: def subgaussian}
  Let $X$ be a random variable. The following conditions are
  equivalent:
  \begin{enumerate}
    \item $X$ is subgaussian;
    \item $\exists a>0 \ \E e^{aX^2}<+\infty$
    \quad \text{\rm ($\psi_2$-condition)};
    \item $\exists B, b>0 \ \forall \l \in \R \ \ \E e^{\l X} \le Be^{\l^2
    b}$
    \quad \text{\rm (Laplace transform condition)};
    \item $\exists K>0 \  \forall p \ge 1 \ \left ( \E |X|^p \right )^{1/p} \le
    K \sqrt{p}$
    \quad \text{\rm (moment condition)}.
  \end{enumerate}
  Moreover, if $X$ is a centered random variable, (3) can be rewritten as
\newline
\hglue .3in $(3)' \ \exists b'>0 \ \forall \l \in \R \ \E e^{\l X}
\le e^{\l^2
    b'}$.
\end{theorem}

\begin{proof}
The proof is a series of elementary calculations. \\
 $(1) \Rightarrow
(2)$ Let $a<v$. By the integral distribution formula,
\[
   \E e^{aX^2} = 1+\int_0^{\infty} 2at e^{at^2} \cdot \P(|X|>t) \, dt
   \le 1+\int_0^{\infty} 2at \cdot C e^{-(v-a)t^2}  \, dt
   < +\infty.
\]
$(2) \Rightarrow (3)$ Let $\l$ be any real number. Then
\[
   \E e^{\l X} =\E e^{\l X - a X^2} e^{aX^2}
   \le \sup_{t \in \R} e^{\l t -at^2} \cdot \E e^{aX^2} \le B e^{\l^2/4a}.
\]
$(3) \Rightarrow (4)$ Set $\l = \sqrt{p}$. Replacing, as before,
the the function by its supremum, we get
\[
  \E|X|^p \le \sup_{t >0} t^p e^{-\sqrt{p} \, t} \cdot \E e^{\sqrt{p} |X|}
  \le \left ( \frac{\sqrt{p}}{e} \right )^p \cdot C e^{pb}.
\]
$(4) \Rightarrow (1)$ Assume first $t \ge e K$. Choose $p$ so that $\frac{K
\sqrt{p}}{t}=e^{-1}$.
\[
  \P (|X|>t) \le \frac{\E |X|^p}{t^p}
  \le \left (\frac{K \sqrt{p}}{t} \right )^p =e^{-p} = e^{-v t^2},
\]
where $v =e^{-2}K^{-2}$.
This proves (1) for $t \ge e K$. Setting $C=e$
 automatically guaranties that (1) holds for
$0<t<eK$ as well. \\
$(3)'$ We will assume that (3) holds with $B>1$ since otherwise the statement is trivial. Assume first that $X$ is symmetric.
 For large values of $\l$, we can derive (3) with constant $B=1$ by changing the parameter $b$.
  Indeed,  set $\l_0=\sqrt{2a}$ and choose  $\bar{b}>0$ so
 that $Be^{\l_0^2b} \le e^{\l_0^2 \bar{b}}$. This guarantees that (3) holds for all $\l$ such that $|\l| \ge \l_0$ with $B=1$ and $b$ replaced by $\bar{b}$.

 If $\l^2 \le 2a$, then by
 Holder's inequality and the $\psi_2$-condition,
 \[
     \E e^{\l X} = \E \frac{1}{2} ( e^{\l X} + e^{-\l X})
     \le \E e^{\l^2 X^2/2}
     \le \left ( \E e^{aX^2} \right )^{\l^2 /2a}
     \le \exp \left (c \frac{\l^2}{2a} \right ).
 \]
  Finally, we set
 $b'= \max(c/2a, \bar{b})$.

 In the general case, we use a simple symmetrization. Let $X'$
 be an independent copy of $X$. Then by Jensen's inequality,
 \[
  \E e^{\l X} = \E e^{\l (X-\E X')} \le \E e^{\l (X-X')},
 \]
 where $X-X'$ is a symmetric subgaussian random variable.
\end{proof}
\begin{remark}
The $\psi_2$-condition turns the set of centered
subgaussian random variables into a normed space. Define the
function $\psi_2: \R \to \R$ by $\psi_2(t)=\exp(t^2)-1$. Then for a
non-zero random variable set
\[
  \norm{X}_{\psi_2} = \inf \{s > 0 \mid \E \psi_2(X/s) \le 1 \}.
\]
The subgaussian random variables equipped with this norm form an
Orlicz space (see \cite{LT} for the details).
\end{remark}

To estimate the first singular value, we have to prove a large
deviation inequality for a linear combination of independent
subgaussian random variables. Note that a linear combination of independent
Gaussian random variables is Gaussian. We prove below that a linear
combination of independent subgaussian random variables is subgaussian.

\begin{theorem}  \label{t: Hoeffding}
  Let $X_1 \etc X_n$ be independent centered subgaussian random variables.
  Then for any $a_1 \etc a_n \in \R$
  \[
     \P \left (\left | \sum_{j=1}^n a_j X_j \right | >t \right )
     \le  2 \exp \left ( - \frac{c t^2}{ \sum_{j=1}^n a_j^2} \right ).
  \]
\end{theorem}
\begin{proof}
  Set $v_j=a_j/\left ( \sum_{j=1}^n a_j^2 \right )^{1/2}$. We have
  to show that the random variable $Y= \sum_{j=1}^n v_j X_j$ is
  subgaussian. Let us check the Laplace transform condition $(3)'$. For any
  $\l \in \R$
\begin{align*}
   &\E \exp \left ( \l \sum_{j=1}^n v_jX_j \right )
    =  \prod_{j=1}^n \E \exp(\l v_j X_j) \\
  &\le \prod_{j=1}^n \exp(\l^2 v_j^2 b)
    = \exp \left (\l^2b \sum_{j=1}^n v_j^2 \right )
    =e^{\l^2 b}.
\end{align*}
The inequality here follows from $(3)'$. Note that the fact that the
constant in front of the exponent in $(3)'$ is 1 plays the crucial
role here.
\end{proof}

Theorem \ref{t: Hoeffding} can be used to give a very short proof of a
classical inequality due to Khinchin.

\begin{theorem}[Khinchin]  \label{t: Khinchin}
 Let $X_1 \etc X_n$ be independent centered subgaussian random variables.
 For any $p \ge 1$ there exist $A_p,B_p>0$ such that the
 inequality
 \[
    A_p \left ( \sum_{j=1}^n a_j^2 \right )^{1/2}
    \le \left ( \E \left | \sum_{j=1}^n a_j X_j \right |^p \right
               )^{1/p}
    \le B_p \left ( \sum_{j=1}^n a_j^2 \right )^{1/2}
 \]
 holds for all $a_1 \etc a_n \in \R$.
\end{theorem}
\begin{proof}
Without loss of generality,  assume that $\left ( \sum_{j=1}^n a_j^2 \right
)^{1/2}=1$.

Let $p \ge 2$. Then by H\"{o}lder's inequality
\[
  \left ( \sum_{j=1}^n a_j^2 \right )^{1/2}
  = \left ( \E \left |\sum_{j=1}^n a_j X_j  \right |^2 \right )^{1/2}
  \le \left ( \left | \E \sum_{j=1}^n a_j X_j \right |^p \right
  )^{1/p},
\]
so $A_p=1$. By Theorem \ref{t: Hoeffding}, $Y= \sum_{j=1}^n a_j X_j $ is
a subgaussian random variable. Hence,
\[
  \left ( \E |Y|^p \right )^{1/p} \le C \sqrt{p}=:B_p.
\]
This is the right asymptotic as $p \to \infty$.

In the case $1 \le p \le 2$ it is enough to prove the inequality for
$p=1$. As before, by H\"{o}lder's inequality, we can choose  $B_p=1$.
Applying Khinchin's inequality with $p=3$, we get
\[
  \E |Y|^2=\E |Y|^{1/2} \cdot |Y|^{3/2}
  \le  \left ( \E |Y| \right )^{1/2}
  \cdot  \left ( \E |Y|^3 \right )^{1/2}
  \le \left ( \E |Y| \right )^{1/2}
  \cdot  B_3^{3/2}\,\left (  \E |Y|^2 \right )^{3/4}.
\]
Hence,
\[
    B_3^{-3} \left ( \E |Y|^2 \right )^{1/2} \le \E |Y|. \qedhere
\]

\end{proof}

\section{Invertibility of a rectangular random matrix} \label{sec: rectangular}

We  introduce the {\em $\e$-net argument}, which will enable us to bound the condition number for a random $N \times n$ matrix with independent entries in the case when $N \gg n$. To simplify the proofs we assume from now on that the entries of the matrix are centered, subgaussian random variables.

Recall the definition  of an $\e$-net.
\begin{definition}
  Let $(T,d)$ be a metric space. Let $K \subset T$. A set
  $\mathcal{N} \subset T$ is called an $\e$-net for $K$ if
  \[
     \forall x \in K \  \exists \, y \in \mathcal{N} \ d(x,y)<\e.
  \]
A set
  $\mathcal{S} \subset K$ is called $\e$-separated if
  \[
     \forall x,y \in  \mathcal{S} \quad d(x,y) \ge \e.
  \]
\end{definition}
 The union of $\e$-balls centered at the $\e$-net $\mathcal{N}$ covers $K$, while the $\e$-balls centered at $\mathcal{S}$ form a packing.
These two notions are closely related.  Namely, we have the
following elementary Lemma.
\begin{lemma}   \label{l: net-separated}
  Let $K$ be a subset of a metric space $(T,d)$, and let
  $\mathcal{N}  \subset T$ be an $\e$-net for $K$. Then
  \begin{enumerate}
    \item there exists a $2\e$-net $\mathcal{N}' \subset K$ such that
      $|\mathcal{N}'| \le |\mathcal{N}|$;
    \item  any $2\e$-separated set $\mathcal{S} \subset K$ satisfies
       $|\mathcal{S}| \le |\mathcal{N}|$.
    \item From the other side, any maximal $\e$-separated set
    $\mathcal{S}' \subset K$ is an $\e$-net for $K$.
  \end{enumerate}
\end{lemma}
  We leave the proof of this lemma for a reader as an exercise.
\begin{lemma}[Volumetric estimate]  \label{l: volumetric}
  For any $\e<1$ there exists an $\e$-net $\mathcal{N} \subset
  S^{n-1}$ such that
  \[
    |\NN| \le \left ( \frac{3}{\e} \right )^n.
  \]
\end{lemma}
\begin{proof}
 Let $\NN$ be a maximal $\e$-separated subset of $S^{n-1}$. Then for
 any distinct points $x,y \in \NN$
 \[
   \left(x+\frac{\e}{2} B_2^n \right) \cap \left(y+\frac{\e}{2} B_2^n\right)= \emptyset.
 \]
 Hence,
 \[
  |\NN| \cdot \vol \left (\frac{\e}{2} B_2^n \right )
  =\vol \left (\bigcup_{x \in \NN} \big(x+\frac{\e}{2} B_2^n\big) \right )
  \le \vol \left ( \big(1 +\frac{\e}{2}\big) B_2^n \right ),
 \]
 which implies
 \[
   |\NN| \le \left( 1+ \frac{2}{\e} \right )^n \le \left(\frac{3}{\e} \right)^n.  \qedhere
 \]
\end{proof}

Using $\e$-nets, we prove a basic bound  on the first singular value
of a random subgaussian matrix:

\begin{proposition}[First singular value]                       \label{p: norm}
  Let $A$ be an $N \times n$ random matrix, $N \ge n$, whose entries
  are independent copies of a centered subgaussian random variable. Then
  $$
  \P \big( s_1(A) > t \sqrt{N} \big)
  \le e^{-c_0 t^2 N}
  \qquad \text{for } t \ge C_0.
  $$
\end{proposition}

\begin{proof}
Let $\NN$ be a $(1/2)$-net in $S^{N-1}$ and $\mathcal{M}$ be a
$(1/2)$-net in $S^{n-1}$.
 For any $u \in S^{n-1}$, we can choose a $x \in \NN$  such that $\norm{x-u}_2 <1/2$.
 Then
 \[
   \norm{Au}_2 \le \norm{Ax}_2+ \norm{A} \cdot \norm{x-u}_2 \le \norm{Ax}_2+ \frac{1}{2} \norm{A}.
 \]
 This shows that $\norm{A} \le 2 \sup_{x \in \NN} \norm{Ax}_2 = 2 \sup_{x \in \NN} \sup_{v \in S^{N-1}} \pr{Ax}{v}$.
 Approximating $v$ in a similar way by an element of $\mathcal{M}$, we obtain
\[
   \norm{A} \le 4 \max_{ x \in \NN, \ y \in \MM} |\pr{A x}{y}|.
\]
 By Lemma \ref{l: volumetric}, we can choose these
nets so that
$$
|\NN|  \le 6^N, \quad |\MM|  \le 6^n.
$$
By Theorem \ref{t: Hoeffding}, for  every $x \in \NN$ and $y \in \MM$,
the random variable $\pr{Ax}{y} =\sum_{j=1}^N \sum_{k=1}^n a_{j,k} y_j x_k $ is subgaussian, i.e.,
$$
\P \big( |\< Ax,y \> | > t \sqrt{N} \big) \le C_1 e^{-c_1 t^2 N}
\qquad \text{for } t > 0.
$$
Taking the union bound, we get
\begin{align*}
\P \big( \|A\| > t \sqrt{N} \big)
 &\le  |\NN| |\MM| \max_{x \in \mathcal{N}, \,
y \in \mathcal{N}}
  \P \big( |\< Ax,y \> |  > t \sqrt{N}/4 \big) \\
&\le  6^N \cdot 6^N \cdot  C_1 e^{-c_2 t^2 N}
\le C_1  e^{-c_0 t^2 N},
\end{align*}
 provided that $t \ge C_0$ for an appropriately chosen constant $C_0>0$.
This completes the proof.
\end{proof}
Proposition \ref{p: norm} means that for any $N \ge n$ the first
singular value is $O(\sqrt{N})$ with probability close to $1$. Thus,
the bound for the condition number reduces to a lower estimate of
the last singular value.

To obtain it, we prove an easy estimate for a small ball probability
of a sum of independent random variables.

\begin{lemma}   \label{l: LPRT}
  Let $\xi_1 \etc \xi_n$ be independent copies of a
  centered subgaussian random variable with variance  $1$.
  Then there exists $\mu \in (0,1)$
  such that for every coefficient vector $a= (a_1,\ldots,a_n) \in S^{n-1}$
  the random sum $S = \sum_{k=1}^n a_k \xi_k$ satisfies
  $$
  \P(|S| < 1/2) \le \mu.
  $$
\end{lemma}

\begin{proof}
 Let $0< \l < (\E S^2)^{1/2}=1$. By the Cauchy--Schwarz inequality,
 \[
  \E S^2 = \E S^2 \mathbf{1}_{[-\l,\l]}(S)+ \E S^2 \mathbf{1}_{\R \setminus [-\l,\l]}(S)
  \le \l^2+ \left(\E S^4 \right)^{1/2} \P (|S| > \l)^{1/2}.
 \]
 This leads to the  Paley--Zygmund inequality:
$$
\P (|S| > \l) \ge \frac{(\E S^2 - \l^2)^2}{\E S^4}
     = \frac{(1-\l^2)^2}{\E S^4}.
$$
 By Theorem \ref{t: Hoeffding}, the random variable $S$ is subgaussian,
 so by part (4) of Theorem~\ref{t: def subgaussian}, $\E S^4 \le C$.
To finish the proof, set $\l=1/2$.
\end{proof}

 Lemma \ref{l: LPRT} implies the following invertibility estimate for a
fixed vector.

\begin{corollary}  \label{c: individual}
 Let $A$ be a matrix as in Proposition~\ref{p: norm}. Assume that all entries of $A$ have variance $1$.
  Then
 there exist constants $\eta,\nu \in (0,1)$
 such that for every $x \in S^{n-1}$,
 \[
   \P ( \norm{Ax}_2  < \eta \sqrt{N}) \le \nu^N.
 \]
\end{corollary}
\begin{proof}
 The coordinates of the vector $Ax$ are independent linear
 combinations of i.i.d. subgaussian random variables with
 coefficients $(x_1 \etc x_n) \in S^{n-1}$. Hence, by Lemma \ref{l:
 LPRT}, $\P(|(Ax)_j|< 1/2) \le \mu$ for all $j =1 \etc N$.

  Assume that $\norm{Ax}_2< \eta \sqrt{N}$. Then $|(Ax)_j|< 1/2$ for
  at least $(1-4\eta^2)N>N/2$ coordinates. If $\eta$ is small enough,
  then the number $M$ of subsets of $\{1 \etc N\}$ with at least
  $(1-4\eta^2)N$ elements is less than $\mu^{-N/4}$. Then the union
  bound implies
  \[
    \P( \norm{Ax}_2  < \eta \sqrt{N})
    \le M \cdot \mu^{N/2} \le \mu^{N/4}.  \qedhere
  \]
\end{proof}
Combining this with the $\e$-net argument, we obtain the
estimate for the smallest singular value of a random matrix, whose
dimensions are significantly different.
\begin{proposition}[Smallest singular value of rectangular matrices]
            \label{p: rectangular}
  Let $A$ be an $N \times n$ matrix whose entries are i.i.d.
  centered subgaussian random variables with variance $1$.
  There exist $c_1, c_2 > 0$ and $\d_0 \in (0,1)$  such that
   if $n < \d_0 N$, then
  \begin{equation}                  \label{eq: rectangular}
    \P \big( \min_{x \in S^{n-1}} \norm{Ax}_2 \le c_1 \sqrt{N}
         \big)
    \le e^{-c_2 N}.
  \end{equation}
\end{proposition}

\begin{proof}
  Let $\e>0$ to be chosen later. Let $\NN$ be an $\e$-net
in $S^{n-1}$  of cardinality $|\NN| \le (3/\e)^n$. Let $\eta$ and
$\nu$ be the numbers in Corollary \ref{c: individual}. Then by the
union bound,
\begin{equation}   \label{on net}
  \P \left ( \exists y \in \NN : \ \norm{Ay}_2 < \eta \sqrt{N} \right )
  \le (3/\e)^n \cdot \nu^N.
\end{equation}
Let $V$ be the event that $\norm{A} \le C_0 \sqrt{N}$ and $\norm{Ay}_2
\ge  \eta \sqrt{N}$ for all points $y \in \NN$.

Assume that $V$ occurs, and let $x \in S^{n-1}$ be any point. Choose
$y \in \NN$ such that $\norm{y-x}_2<\e$. Then
\[
  \norm{Ax}_2 \ge \norm{Ay}_2 - \norm{A} \cdot \norm{x-y}_2
 \ge  \eta \sqrt{N} - C_0 \sqrt{N} \cdot \e
 = \frac{\eta \sqrt{N}}{2},
\]
if we set $\e=\eta/(2C_0)$. By \eqref{on net} and Proposition \ref{p:
norm},
$$
    \P (V^c) \le \big( \nu \cdot  \left ( 3/\e \right
    )^{n/N} \big)^N + e^{-c'N}
    \le e^{-c_2 N},
$$
if we assume that $n/N \le \d_0$ for an appropriately chosen $\d_0 <
1$. This completes the proof.
\end{proof}

\begin{remark}
 Note that although we assumed that the entries of the matrix $A$ are independent,
 Proposition \ref{p: rectangular} can be proved under a weaker assumption.
 It is enough to assume that
 for any $x \in S^{n-1}$, the coordinates of the vector $Ax$ are independent centered subgaussian
 random variables of unit variance.
 Indeed, in this case Corollary~\ref{c: individual} applies without any changes.
 We will use this observation in Subsection~\ref{ss: p>2}.
\end{remark}

\section{Invertibility of a square matrix: \\
absolutely continuous entries}  \label{sec: continuous}

 Until recently, much less has been known about the behavior of the smallest singular
value of a square matrix. In the classic work on numerical inversion
of large matrices, von~Neumann and his associates used random
matrices to test their algorithms, and they speculated that
\begin{equation}                \label{whp}
  s_n(A) \sim n^{-1/2} \quad \text{with high probability}
\end{equation}
(see \cite{vN}, pp.~14, 477, 555). In a more precise form, this
estimate was conjectured by Smale \cite{S 85} and proved by Edelman
\cite{E 88}  and Szarek \cite{Sz} for {\em random Gaussian matrices}
$A$, i.e., those with i.i.d. standard normal entries. Edelman's theorem
states that for every $\e \in (0,1)$,
\begin{equation}                \label{tail}
  \P \big( s_n(A) \le \e n^{-1/2} \big) \sim \e.
\end{equation}

Conjecture \eqref{whp} for general random matrices was an open
problem, unknown even for the {\em random sign matrices} $A$, i.e., those
whose entries are $\pm 1$ symmetric random variables. The first polynomial bound for
the smallest singular value of a random matrix with i.i.d. subgaussian, in particular, $\pm 1$ entries was obtained in \cite{R square}. It was proved that for such matrix $s_n(A) \ge C n^{-3/2}$ with high probability. Following that, Tao and Vu proved that if $A$ is a $\pm 1$ random matrix, then for any $\a>0$ there exists $\b>0$ such that $s_n(A) \ge n^{-\b}$ with probability at least $1-n^{-\a}$.
In
\cite{RV-square}
 the conjecture \eqref{whp} is proved in full generality under the
 fourth moment assumption.

\begin{theorem}[Invertibility: fourth moment]       \label{t: 4}
  Let $A$ be an $n \times n$ matrix whose entries are independent centered
  real random variables with variances at least $1$ and fourth moments
  bounded by $B$.
  Then, for every $\d > 0$ there exist $\e > 0$ and $n_0$
  which depend (polynomially) only on $\d$ and $B$,  such that
  $$
  \P \big( s_n(A) \le \e n^{-1/2} \big) \le \d
    \qquad \text{for all $n \ge n_0$}.
  $$
\end{theorem}

This shows in particular that the median of $s_n(A)$ is at least of order
$n^{-1/2}$. To show that $s_n(A) \sim n^{-1/2}$ with high probability, one has to prove a matching lower bound.
This was done in \cite{RV upper bound} for matrices with subgaussian entries and extended in \cite{V product} to matrices, whose entries have the finite fourth moment.

\medskip

Under stronger moment assumptions, more is known about the
distribution of the largest singular value, and similarly one hopes
to know more about the smallest singular value.

One might then expect that the estimate \eqref{tail} for the
distribution of the smallest singular value of Gaussian matrices
should hold for all subgaussian matrices. Note however that
\eqref{tail} fails for the random sign matrices, since they are
singular with positive probability. Estimating the
probability of singularity for random sign matrices is a longstanding open problem.
Even proving that it converges to $0$ as $n \to \infty$ is a
nontrivial result due to Koml\'os \cite{K 67}. Later Kahn, Koml\'os
and Szemer\'edi \cite{KKS} showed that it is exponentially small:
\begin{equation}            \label{KKS}
  \P \big( \text{random sign matrix $A$ is singular} \big) < c^n
\end{equation}
for some universal constant $c \in (0,1)$. The often conjectured
optimal value of $c$ is $1/2 + o(1)$ \cite{KKS}, and the best known
value $1/\sqrt{2} + o(1)$ is due to Bourgain, Vu, and Wood \cite{BVW}, (see \cite{TV det, TV
singularity} for earlier results).

Spielman and Teng \cite{ST} conjectured that \eqref{tail} should
hold for the random sign matrices up to an exponentially small term
that accounts for their singularity probability:
$$
\P \big( s_n(A) \le \e n^{-1/2} \big) \le \e + c^n.
$$

We prove Spielman-Teng's conjecture up to a coefficient in front of $\e$. Moreover, we show that this type of behavior is common  for all matrices with
subgaussian i.i.d. entries.
 For a bound for random matrices with general i.i.d. entries  see
\cite{RV-square}.

\begin{theorem}[Invertibility: subgaussian]   \label{t: subgaussian}
  Let $A$ be an $n \times n$ matrix whose
  entries are independent copies of a centered subgaussian real random
  variable.
  Then for every $\e \ge 0$, one has
  \begin{equation}                      \label{eq main}
    \P \big( s_n(A) \le \e n^{-1/2} \big)
    \le C\e +  c^n,
  \end{equation}
  where $C > 0$  and $c \in (0,1)$.
\end{theorem}
Note that setting $\e=0$ we recover the result of Kahn, Koml\'os and
Szemer\'edi.  Also, note that the question whether \eqref{eq main} holds for random sign matrices with coefficient $C=1$ remains open.

We shall start with an attempt to apply the $\e$-net argument. Let
us consider an $n \times n$ Gaussian matrix, i.e., a matrix with
independent $N(0,1)$ entries. In this case, for any $x \in S^{n-1}$,
the vector $Ax$ has independent $N(0,1)$ coordinates, so it is
distributed like the standard Gaussian vector in $\R^n$. Hence, for
any $t>0$,
\begin{align*}
  \P(\norm{Ax}_2 \le t \sqrt{n})
  &= (2 \pi)^{-n/2} \int_{t \sqrt{n} \cdot B_2^n} e^{-\norm{x}_2^2/2}
  \, dx
  \le (2 \pi)^{-n/2} \vol(t \sqrt{n} \cdot B_2^n)  \\
  &\le (C_1 t)^n.
\end{align*}
 Fix $\e>0$. Let $\NN$ be an $\e$-net
in $S^{n-1}$  of cardinality $|\NN| \le (3/\e)^n$.  Then by the
union bound,
\[
  \P \left ( \exists x \in \NN : \ \norm{Ax}_2 < t n^{1/2} \right )
  \le (3/\e)^n \cdot (C_1 t)^n.
\]
To obtain a meaningful estimate we have to require
\begin{equation}\label{i: t<epsilon}
    (3/\e) \cdot (C_1 t) <1.
\end{equation}
As in Proposition \ref{p: rectangular}, we may assume that $\norm{A}
\le C_0 \sqrt{n}$, since the complement of this event has an
exponentially small probability. Assume that for any $y \in \NN, \
\norm{Ay}_2 \ge t \sqrt{n}$. Given $x \in S^{n-1}$, find $y \in \NN$
satisfying $\norm{x-y}_2<\e$. Then
\[
  \norm{Ax}_2 \ge \norm{Ay}_2 - \norm{A} \cdot \norm{x-y}_2
 \ge  t n^{1/2} - C_0 n^{1/2} \cdot \e.
\]
To obtain a non-trivial lower bound, we have to assume that
\begin{equation}\label{i: epsilon<t}
    t > C_0 \e.
\end{equation}
Unfortunately, the system of inequalities \eqref{i: t<epsilon} and
\eqref{i: epsilon<t} turns out to be inconsistent, and the $\e$-net argument
fails for the square matrix. Nevertheless, a part of this idea can
be salvaged. Namely, if the cardinality of the $\e$-net satisfies a
better estimate
\begin{equation}\label{i: small net}
    |\NN| \le (\a/\e)^n
\end{equation}
for a small constant $\a>0$, then \eqref{i: t<epsilon} is replaced
by $(\a/\e) \cdot (C_1 t) <1$, and the system \eqref{i: t<epsilon},
\eqref{i: epsilon<t} becomes consistent. Although the estimate
\eqref{i: small net} is impossible for the whole sphere, it can be
obtained for a small part of it. This becomes the first ingredient
of our strategy: small parts of the sphere will be handled by the
$\e$-net argument. However, the ``bulk'' of the sphere has to be
handled differently.

The proof of Theorem \ref{t: subgaussian} for random matrices with
 i.i.d. subgaussian entries having a bounded density is presented below.

\subsection{Conditional argument}
To handle the ``bulk'', we have to produce an estimate which holds
for all vectors in it simultaneously, without taking the union
bound. Let $x \in S^{n-1}$ be a vector such that $|x_1|\ge
n^{-1/2}$. Denote the columns of the matrix $A$ by $X_1 \etc X_n$,
and let
\[
  H_j: =\text{span} (X_k \mid k \neq j).
\]
Then $Ax= \sum_{k=1}^n x_k X_k$, so
\begin{equation}  \label{i: norm-distance}
  \norm{Ax}_2 \ge \dist (Ax, H_1)= \dist (x_1 X_1, H_1)
  \ge n^{-1/2} \dist (X_1,H_1).
\end{equation}
Note that the right hand side is independent of $x$. Therefore it
provides a uniform lower bound for all $x$ such that $|x_1| \ge
n^{-1/2}$. Since any vector $x \in S^{n-1}$ has a coordinate with
absolute value greater than $n^{-1/2}$, we can try to extend this
bound to the whole sphere. This approach immediately runs into a
problem: we don't know {\em a priori} which of the coordinates of
$x$ is big. To modify this approach we shall pick a {\em random}
coordinate. To this end we have to know that the random coordinate
is big with relatively high probability. This is true for vectors,
which look like the vertices of a discrete cube, but is obviously
false for vectors with small support, i.e. a small number of non-zero coordinates. This observation leads us to
the first decomposition of the sphere:
\begin{definition}[Compressible and incompressible vectors]   \label{d: compressible}
  Fix $\d, \rho \in (0,1)$.
  A vector $x \in \R^n$ is called {\em sparse} if
  $|\supp(x)| \le \d n$. (Here $\supp(x)$ means the set of non-zero coordinates of $x$.)
  A vector $x \in S^{n-1}$ is called {\em compressible} if $x$
  is within Euclidean distance $\rho$ from the set of
  all sparse vectors.
  A vector $x \in S^{n-1}$ is called {\em incompressible}
  if it is not compressible.
  The sets of sparse, compressible and incompressible vectors
  will be denoted by $\Sparse$, $\Comp$
  and $\Incomp$ respectively.
\end{definition}
Using the decomposition of the sphere $S^{n-1} = \Comp \cup
\Incomp$, we break the invertibility problem into two subproblems,
for compressible and incompressible vectors:
\begin{multline}                        \label{i: two terms}
\P \big( s_n(A) \le \e n^{-1/2}  \big) \\
  \le \P \big( \inf_{x \in \Comp} \|A x\|_2 \le \e n^{-1/2}
        \big) \\
    + \P \big( \inf_{x \in \Incomp} \|A x\|_2 \le \e n^{-1/2}
        \big).
\end{multline}

On the set of compressible vectors, we obtain an inequality, which is much stronger than we need.

\begin{lemma}[Invertibility for compressible vectors]       \label{l: compressible}
  Let $A$ be a random matrix as in Theorem~\ref{t: subgaussian},
  Then there exist $\d, \rho, c_1, c_2 > 0$  such that
  $$
  \P \big( \inf_{x \in \Comp} \|A x\|_2 \le c_1 n^{1/2}
     \big)
  \le e^{-c_2 n}.
  $$
\end{lemma}
\begin{proof}[Sketch of the proof]
 Any compressible vectors is close to a coordinate subspace of a
small dimension $\d n$. The restriction of our random matrix $A$
onto such a subspace is a random {\em rectangular} $n \times \d n$
matrix. Such matrices are well invertible outside of an event of exponentially small
probability, provided that $\d$ is small enough (see Proposition~\ref{p: rectangular}). By taking the
union bound over all coordinate subspaces, we  deduce the
invertibility of the random matrix on the set of compressible
vectors.
\end{proof}
We shall fix $\d$ and $\rho$ as in Lemma \ref{l: compressible} for
the rest of the proof.

The incompressible vectors are well spread in the sense that they
have many coordinates of the order $n^{-1/2}$. This observation will
allow us to realize the scheme described at the beginning of this
section.

\begin{lemma}[Incompressible vectors are spread]    \label{l: spread}
  Let $x \in \Incomp$.
  Then there exists a set $\s(x) \subseteq \{1, \ldots, n\}$
  of cardinality $|\s(x)| \ge \nu_1 n$ and such that
  $$
  \frac{\nu_2}{\sqrt{n}} \le |x_k| \le \frac{\nu_3}{\sqrt{ n}}
  \qquad \text{for all $k \in \s$.}
  $$
  Here $0< \nu_1, \nu_2<1$ and $\nu_3>1$ are  constants depending only on the parameters $\d, \rho$.
\end{lemma}
 We leave the proof of this lemma to the reader.

The main difficulty in implementing the distance bound like
\eqref{i: norm-distance} is to avoid taking the union bound. We
achieve this in the proof of the next lemma by a random choice of a  coordinate.
\begin{lemma}[Invertibility via distance]  \label{l: via distance}
  Let $A$ be a random matrix with i.i.d. entries.
  Let $X_1,\ldots,X_n$ denote the column vectors of $A$, and let $H_k$ denote
  the span of all column vectors except the $k$-th one:
$
  H_j =\text{span} (X_k \mid k \neq j).
$
  Then for  every $\e > 0$, one has
  \begin{equation}                              \label{eq: via distance}
    \P \big( \inf_{x \in \Incomp} \|A x\|_2 < \e \nu_2 n^{-1/2} \big)
    \le \frac{1}{\nu_1} \cdot
          \P \big( \dist( X_n, H_n) < \e \big).
  \end{equation}
\end{lemma}

\begin{proof}
Denote
$$
p := \P \big( \dist( X_k, H_k) < \e \big).
$$
 Note that since the entries of the matrix $A$ are i.i.d., this
 probability does not depend on $k$.
Then
$$
\E \big| \{ k :\, \dist(X_k, H_k) < \e \} \big|
  = n p.
$$
Denote by $U$ the event that the set $\s_1 := \{ k :\, \dist(X_k,
H_k) \ge \e \}$ contains more than $(1-\nu_1)n$ elements. Then by
Chebychev's inequality,
$$
\P(U^c) \le \frac{p}{\nu_1}.
$$
Assume that the event $U$ occurs.  Fix any incompressible vector $x$
and let $\s(x)$ be the set from Lemma \ref{l: spread}.
 Then $|\s_1| + |\s(x)| > (1-\nu_1)n + \nu_1 n = n$, so the sets $\s_1$
and $\s(x)$ have nonempty intersection. Let $k \in \s_1 \cap \s(x)$,
so
\[
  |x_k| \ge \nu_2 n^{-1/2} \quad \text{and} \quad
  \dist (X_k, H_k) \ge  \e.
\]
 Writing $A x = \sum_{j=1}^n x_j X_j$, we get
\begin{align*}
\|Ax\|_2
  &\ge  \dist (Ax, H_k)
  =  \dist (x_k X_k, H_k)
  =  |x_k| \, \dist (X_k, H_k) \\
  &\ge \nu_2 n^{-1/2} \cdot \e.
\end{align*}

Summarizing, we have shown that
$$
\P \big( \inf_{x \in \Incomp} \|A x\|_2 < \e \nu_2 n^{-1/2} \big)
  \le \P(U^c)
  \le \frac{p}{\nu_1}.
$$
This completes the proof.
\end{proof}

Lemma~\ref{l: via distance} reduces the invertibility problem to a
lower bound on the distance between a random vector and a random
subspace. Now we reduce bounding the distance to a small ball
probability estimate.

Let $X_1,\ldots,X_n$ be the column vectors of $A$. Let $Z$ be any
unit vector orthogonal to $X_1,\ldots, X_{n-1}$. We call it a {\em
random normal}.
We clearly have
\begin{equation}                                    \label{i: dist}
  \dist(X_n,H_n) \ge |\< Z, X_n\> |.
\end{equation}
The vector $Z$ depends only on $X_1 \etc X_{n-1}$,  so $Z =: (a_1,\ldots,a_n)$  and $X_n =:
(\xi_1,\ldots,\xi_n)$ are independent.
Condition on the vectors $X_1
\etc X_{n-1}$. Then the vector $Z$ can be viewed as fixed, and the
problem reduces to the small ball probability estimate for a linear
combination of independent random variables
$$
\< Z, X_n\> = \sum_{k=1}^n a_k \xi_k.
$$

Assume for a moment that the distribution of a random variable $\xi$ is absolutely
continuous with bounded density. Then
\begin{equation}  \label{i: all t}
  \P(|\xi|<t) \le C't \quad \text{for any } t>0.
\end{equation}
This estimate can be extended to a linear combination of independent
copies of $\xi$. Therefore,
\[
  \P (|\pr{Z}{X_n}|< t \mid Z) \le Ct.
\]
Integrating over $X_1 \etc X_{n-1}$, we obtain
\[
  \P (|\pr{Z}{X_n}|< t) \le Ct.
\]
Thus, combining this estimate with Lemma \ref{l: via distance}, we
prove that
\[
    \P \big( \inf_{x \in \Incomp} \|A x\|_2 < \e \nu_2 n^{-1/2} \big)
    \le C \e.
\]
Then \eqref{i: two terms} and Lemma \ref{l: compressible} imply
Theorem \ref{t: subgaussian} in this case even without the additive term $c^n$.

\section{Arithmetic structure and the small ball probability}  \label{sec: small ball}
To prove Theorem \ref{t: subgaussian} in the previous section, we
used the small ball probability estimate \eqref{i: all t}. However,
this estimate does not hold for a general subgaussian random
variable, and in particular for any random variable having an atom
at $0$.

Despite this, a linear combination $\sum_{k=1}^n a_k \xi_k$ of
independent copies of a subgaussian random variable $\xi$ obeys an
estimate similar to \eqref{i: all t} for a {\em typical} vector
$a=(a_1 \etc a_n)$ up to a certain threshold. It is easy to see that
this threshold should depend on the vector $a \in S^{n-1}$. Indeed, assume that
$\xi$ is the random $\pm 1$ variable. Then for
\[
  a^{(1)}= \left (\frac{1}{\sqrt{2}}, \frac{1}{\sqrt{2}}, 0 \etc 0 \right ),
   \qquad
  \P \left( \sum_{k=1}^n a_k \xi_k=0 \right)= \frac{1}{2}.
\]
This singular behavior is due to the fact that the vector $a^{(1)}$
is sparse. If we choose the vector $a$, which is far from the sparse
ones, i.e. an incompressible vector, the small ball probability may
be significantly improved. Consider for example, the vector
\[
  a^{(2)}= \left (\frac{1}{\sqrt{n}}, \frac{1}{\sqrt{n}} \etc
    \frac{1}{\sqrt{n}} \right ).
\]
Then by the Berry--Ess\'{e}en Theorem,
\[
  \P \left ( \left |
    \sum_{k=1}^n \frac{1}{\sqrt{n}} \xi_k \right | \le t \right )
  \le C \left (t+ \frac{1}{\sqrt{n}} \right )
\]
This estimate cannot be improved, since for an even $n$,
\[
  \P \left ( \sum_{k=1}^n \frac{1}{\sqrt{n}} \xi_k=0 \right )
  \ge \frac{c}{\sqrt{n}}.
\]
The  coordinates of the vector $a^{(2)}$ are the same, which results
in a lot of cancelations in the random sum $\sum_{k=1}^n a_k \xi_k$.
If the arithmetic structure of the coordinates of the vector $a$ is
less rigid, the small ball  probability can be improved even
further. For example, for the (not normalized) vector
\[
  a^{(3)}= \left (\frac{1+1/n}{\sqrt{n}}, \frac{1+2/n}{\sqrt{n}}
                   \etc \frac{1+n/n}{\sqrt{n}} \right ),
   \qquad
  \P \left( \sum_{k=1}^n a_k \xi_k=0 \right) \sim n^{-3/2}.
\]

 Determining the influence of the arithmetic structure of
the coordinates of a vector $a$ on the small ball probability for
the random sum $\sum_{k=1}^n a_k \xi_k$ became known as the {\em
Littlewood--Offord Problem}. It was investigated by
Littlewood and Offord \cite{LO}, Erd\"{o}s \cite{E 45}, S\'{a}rc\"{o}zy and Szem\'{e}redi \cite{SSz},
etc. Recently Tao and Vu \cite{TV} put forward the inverse
Littlewood--Offord theorems, stating that the large value of the
small ball probability implies a rigid arithmetic structure.
 The inverse Littlewood--Offord theorems are extensively discussed in \cite{TV book}, see also \cite{NV LO} for current results in this direction.
 We will
need a result of this type for the conditional argument to
compensate for the lack of the bound \eqref{i: all t}.

The additive structure of a sequence $a = (a_1,\ldots,a_n)$ of real
numbers $a_k$ can be described in terms of the shortest arithmetic
progression into which it  embeds. This length is conveniently
expressed as the  least common denominator of $a$, defined as
follows:
\[
  \text{lcd}(a) := \inf \Big\{ \theta > 0: \
       \theta a \in \Z^n \setminus \{0\}  \Big\}.
\]

For the vector $a^{(2)}$,
\[
  \text{lcd}(a^{(2)})=\sqrt{n} \sim 1 \Big/\P  \left( \sum_{k=1}^n a_k
  \xi_k=0 \right).
\]
A similar phenomenon occurs for the vector $a^{(3)}$:
\[
  \text{lcd}(a^{(3)})=n^{3/2} \sim 1 \Big/\P  \left( \sum_{k=1}^n a_k
  \xi_k=0 \right).
\]
This suggests that the least common denominator of the sequence
controls the small ball probability. However, in the case when $t
>0$, or when the random variable $\xi$ is not purely discrete, the
precise inclusion $\theta a \in \Z^n\setminus \{0\}$ loses its
meaning. It should be relaxed to measure the closeness of the vector
$\theta a$ to the integer lattice. This leads us to the definition
of the {\em essential least common denominator}.

 Fix a parameter $\g \in (0,1)$.
 For $\a>0$   define
$$
  \LCD_{\a}(a)
  := \inf \Big\{ \theta > 0: \; \dist (\theta a, \Z^n) <
         \min(\g\|\theta a\|_2,\a) \Big\}.
$$
The requirement that the distance is smaller than $\g\|\theta a\|_2$
forces us to consider only non-trivial integer points as approximations
of $\theta a$ -- only those in a small aperture cone around the
direction of $a$ (see the picture below).

  \begin{center}
            \includegraphics[height=7cm]{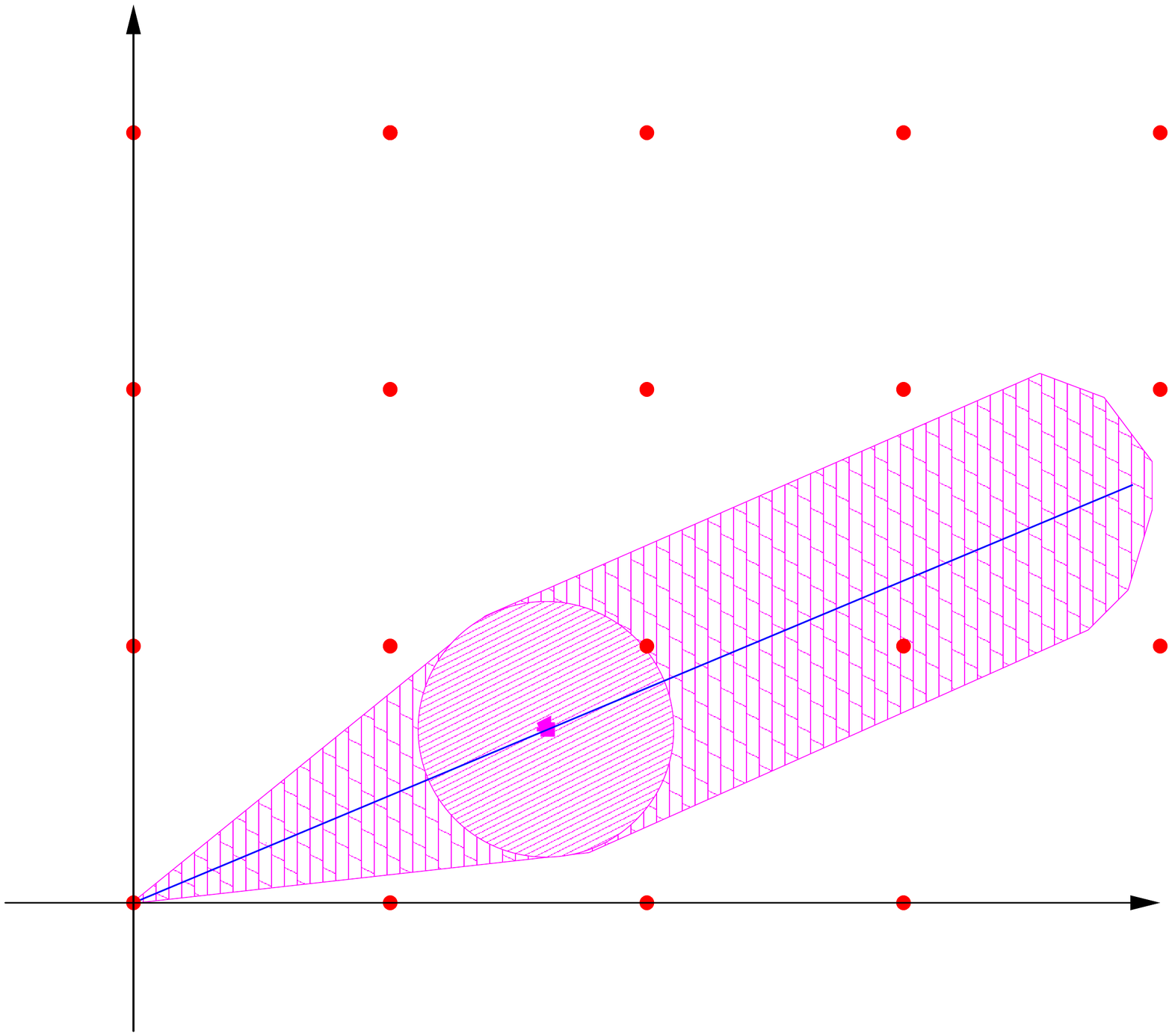}
  \end{center}

 One typically uses this definition with $\g$ a
small constant, and for $\a = c \sqrt{n}$ with a small constant
$c>0$. The inequality $\dist(\theta a, \Z^n) < \a$ then yields that
most coordinates of $\theta a$ are within a small constant distance
from integers. This choice would allow us to conclude that the least
common denominator of any incompressible vector is of order at least $\sqrt{n}$.
Let us formulate this statement precisely.
\begin{lemma}  \label{l: incomp LCD}
   There exist constants $\g>0$ and $\l>0$ depending only on the compressibility parameters $\d, \rho$
   such that any incompressible vector $a$ satisfies $\LCD_{\a}(a) \ge \l \sqrt{n}$.
\end{lemma}
\begin{proof}
 Assume that $a$ is an incompressible vector, and let $\s(a)$ be the set defined in Lemma \ref{l: spread}.
 If $\LCD_{\a}(a) <\l \sqrt{n}$, then
 \[
  \norm{\theta a - z}_2< \g \theta< \g \l \sqrt{n} \qquad \text{for some } \theta \in (0,\l \sqrt{n}), \ z \in \Z^n.
 \]
 Let $I(a)$ be the set of all $j\in \{1 \etc n\}$ such that
 \[
   |\theta a_j-z_j|< \frac{2\g \l}{\nu_1}.
 \]
 The previous inequality implies that $|I(a)|>(1-\nu_1/2)n$. Therefore, for the set $J(a)=I(a) \cap \s(a)$, we have
 \[
  |J(a)| > \frac{\nu_1}{2} n.
 \]
 For any $j \in J(a)$, we have
 \[
  |z_j| < \theta |a_j|+ \frac{2\g \l}{\nu_1}
  < \l \sqrt{n} \cdot \frac{\nu_3}{\sqrt{n}} +\frac{2\g \l}{\nu_1} <1,
 \]
 provided that $\l$ is chosen so that $\l \left(\nu_3+\frac{2\g}{\nu_1} \right)<1$.
 Since $z \in \Z$, this means that $z_j=0$. Finally, this implies
 \[
  \norm{\theta a-z}_2 \ge \left( \sum_{j \in J(a)} \theta^2 a_j^2 \right)^{1/2}
  >\theta \nu_2 \sqrt{\frac{\nu_1}{2}}> \g \theta
 \]
 for $\g < \nu_2 \sqrt{\nu_1/2}$.
 This contradicts the assumption that $\LCD_{\a}(a) <\l \sqrt{n}$.
\end{proof}

We fix $\g$ satisfying Lemma \ref{l: incomp LCD} for the rest of the proof.

The following theorem gives a bound on the small ball probability
for a random sum  in terms of the additive structure of $a$. The
less structure $a$ has, the bigger its least common denominator is,
and the smaller  the small ball probability is.

\begin{theorem}[Small ball probability]                         \label{t: SBP}
  Let $\xi_1, \ldots, \xi_n$ be independent copies of a  centered
  subgaussian random variable $\xi$ of unit variance.
  Consider a sequence $a = (a_1,\ldots,a_n) \in S^{n-1}$.
  Then, for every $\a > 0$, and for
  $$
    \e \ge \frac{(4/\pi)}{\LCD_{\a}(a)},
  $$
  we have
  $$
    \P \left ( \left | \sum_{k=1}^n a_k \xi_k \right |
               \le \e \right )
  \le C \e + C e^{- c\a^2}.
  $$
\end{theorem}

We shall prove more than is claimed in the Theorem. Instead of the
small ball probability we shall bound  a parameter, which controls
the concentration of a random variable around any fixed point.
\begin{definition}
  The {\em L\'evy concentration function} of a random variable $S$
  is defined for $\e > 0$ as
  $$
  \LL(S, \e) = \sup_{v \in \R} \P ( |S-v| \le \e ).
  $$
\end{definition}

The proof of the Theorem uses the Fourier-analytic approach
developed by Hal\'{a}sz \cite{H}, \cite{H 75}.

We start with the classical Lemma of Ess\'{e}en, which estimate the
L\'{e}vy concentration function in terms of the characteristic function
of a random variable.
\begin{lemma}  \label{l: Esseen}
  Let $Y$ be a real-valued random variable. Then
  \[
    \sup_{v \in \R} \P (|Y-v| \le 1) \le C \int_{-2}^{2}
    |\phi_Y(\theta)| \, d \theta,
  \]
  where $\phi_Y(\theta)=\E \exp ( i \theta Y)$ is the
  characteristic function of $Y$.
\end{lemma}

\begin{proof}
  Let $\psi= \chi_{[-1,1]} * \chi_{[-1,1]}$ and let $f=\hat{\psi}$:
  \[
    f(t)= \left ( \frac{2 \sin t}{t} \right )^2.
  \]
  Then both $f \in L_1(\R)$ and $\psi \in L_1(\R)$, so $f$ satisfies the Fourier
  inversion formula. Note also, that $f(t)\ge c$ whenever $|t| \le
  1$. Therefore,
  \begin{align*}
    &\P(|X-v | \le 1)
    =\E \chi_{[-1,1]}(X-v)
     \le \frac{1}{c} \E f(X-v) \\
    &= \frac{1}{c} \E \left ( \frac{1}{2 \pi}
       \int_{\R} \psi(\theta) e^{i \theta (X-v)} \, d\theta \right )
    \le  \frac{1}{2 \pi c}
       \int_{\R} \psi(\theta) | \E e^{i \theta (X-v)}| \, d\theta \\
    &\le \frac{1}{ \pi c}
       \int_{-2}^2 | \E e^{i \theta X}| \, d\theta.
  \end{align*}
 The last inequality follows from $\text{supp}(\psi) = [-2,2]$ and
 $\psi(x) \le 2$.
\end{proof}

\begin{proof}[Proof of Theorem \ref{t: SBP}]
To make the proof more transparent, we shall assume that $\xi$ is the random
$\pm1$ variable. The general case is considered in \cite{RV-square}.

Let $S= \sum_{j=1}^n a_j \xi_j$. Applying Ess\'{e}en's Lemma to the
random variable $Y=S/\e$, we obtain
\begin{equation}  \label{i: Esseen}
  \LL(S,\e) \le C \int_{-2}^{2}
    |\phi_S(\theta/\e)| \, d \theta
  =C \int_{-2}^{2}
    \prod_{j=1}^n |\phi_j(\theta/\e)| \, d \theta,
\end{equation}
where
\[
  \phi_j(t)=\E \exp( ia_j \xi_j t)=\cos ( a_j t).
\]
 The last equality in \eqref{i: Esseen} follows from the
 independence of $\xi_j, \ j=1 \etc n$.
The inequality $|x| \le \exp(-\frac{1}{2}(1-x^2))$, which is valid
for all $x \in \R$, implies
\[
  |\phi_j(t)|
  \le \exp \left (- \frac{1}{2} \sin^2 ( a_j t) \right )
  \le \exp \left ( -\frac{1}{2} \min_{q \in \Z} |\frac{2}{\pi} a_j t -q|^2 \right ).
\]
 In the last inequality we estimated the absolute value of the sinus by a piecewise linear function, see the picture below.

  \begin{center}
             \includegraphics[height=3cm]{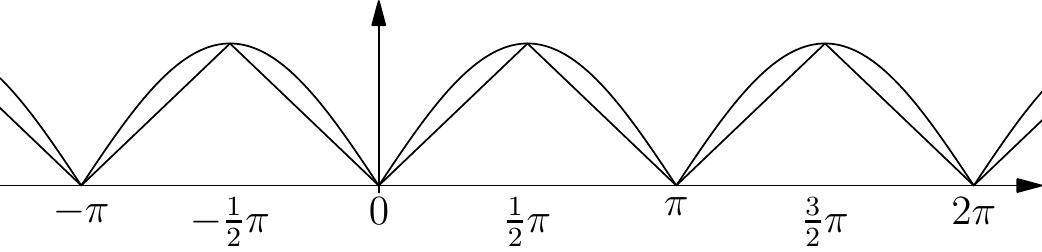}
  \end{center}

Combining the previous inequalities, we get
\begin{align}  \label{i: LL}
  \LL(S,\e)
  &\le C \int_{-2}^2 \exp \left (- \frac{1}{2}
      \sum_{j=1}^n  \min_{q \in \Z}\left |\frac{2}{\pi} a_j \cdot
           \frac{\theta}{\e} -q \right |^2
                         \right  ) \, d\theta \\
  &= C \int_{-2}^{2}
    \exp(-h^2(\theta)/2) \, d \theta, \notag
\end{align}
where
\[
  h(\theta)=\min_{p \in \Z^n} \norm{\frac{2}{\pi \e}\cdot \theta
a-p}_2.
\]

Since by the assumption, $4/(\pi \e) \le \LCD_{\a}(a)$, the definition of the least
common denominator implies that for any $\theta \in [-2,2]$,
\[
  h(\theta)
  \ge \min (\g \frac{2 }{\pi \e} \cdot \theta \norm{a}_2, \a).
\]
Recall that $\norm{a}_2=1$. Then  the previous inequality implies
\[
    \exp(-h^2(\theta)/2)\le \exp \left ( - \left (\frac{2 \g}{\pi \e}
    \theta \right )^2 \Big / 2 \right ) +\exp (-\a^2/2).
\]
Substituting this into \eqref{i: LL} we complete the proof.
\end{proof}

To apply the previous result for random matrices we shall combine it
with the following {\em Tensorization Lemma}.
\begin{lemma}[Tensorization]   \label{l: tensorization}
  Let $\z_1, \ldots, \z_m$ be independent real random
  variables, and let $K, \e_0 \ge 0$.
   Assume that for each $k$
  $$
  \P (|\z_k| < \e) \le K \e
  \qquad \text{for all $\e \ge \e_0$}.
  $$
  Then
  $$
  \P \Big( \sum_{k=1}^m \z_k^2 < \e^2 m \Big) \le (C K \e)^m
  \qquad \text{for all $\e \ge \e_0$},
  $$
  where $C$ is an absolute constant.

\end{lemma}

\begin{proof}
Let $\e \ge \e_0$. We have
\begin{align}                              \label{i: tensorization product}
\P \Big( \sum_{k=1}^m \z_k^2 < \e^2 m \Big)
  &= \P \Big( m - \frac{1}{\e^2} \sum_{k=1}^m \z_k^2 > 0 \Big)
  \le \E \exp \Big( m - \frac{1}{\e^2} \sum_{k=1}^m \z_k^2 \Big) \notag \\
  &= e^m \prod_{k=1}^m \E \exp(-\z_k^2/\e^2).
\end{align}
By Fubini's theorem,
\[
    \E \exp(-\z_k^2/\e^2)
    = \E \int_{|\z|/\e}^{\infty} 2u e^{-u^2} \, du
    = \int_0^\infty 2 u e^{-u^2} \, \P(|\z_k|/\e < u) \; du.
\]
For $u \in (0,1)$, we have $\P(|\z_k| /\e <  u) \le \P(|\z_k| < \e) \le K
\e$. This and the assumption of the lemma yields
$$
\E \exp(-\z_k^2/\e^2) \le \int_0^1 2 u e^{-u^2} K \e \; du
  + \int_1^\infty 2 u e^{-u^2} K \e u \; du
\le  C K \e.
$$

Putting this into \eqref{i: tensorization product} yields
$$
\P \Big( \sum_{k=1}^m \z_k^2 < \e^2 m \Big) \le e^m ( C K \e)^m.
$$
This completes the proof.
\end{proof}
Combining Theorem \ref{t: SBP} and Lemma \ref{l: tensorization}
yields the multidimensional small ball probability estimate similar
to the one we had for absolutely continuous random variable.

\begin{lemma}[Invertibility on a single vector]             \label{l: single}
   Let
  $A'$ be an $m \times n$ random matrix, whose entries are independent
  copies of a centered subgaussian random variable $\xi$ of unit variance.
  Then for any $\a>0$, for every vector $x \in S^{n-1}$, and for every $t \ge 0$,
  satisfying
  \[
    t \ge \max \left(\frac{(4/\pi)}{\LCD_\a(x)} , \ e^{-c \a^2} \right),
  \]
  one has
  $$
  \P \big( \|A'x\|_2 < t n^{1/2} \big)
  \le ( C t )^{m}.
  $$
\end{lemma}
To prove Lemma \ref{l: single}, note that $\norm{A'x}_2^2$ can be represented as
$\norm{A'x}_2^2 = \sum_{k=1}^{m} \z_k^2$, where $\z_k=\sum_{j=1}^n a'_{k,j} x_j$
 are i.i.d. random variables satisfying the conditions of Theorem \ref{t: SBP}.

\section{Putting all ingredients together}  \label{sec: random normal}

Now we have developed all necessary tools to prove the invertibility
theorem, which has been formulated in Section  \ref{sec: continuous}.

\begin{theorem*}{\rm  \ref{t: subgaussian}.}
  Let $A$ be an $n \times n$ matrix whose
  entries are independent copies of a centered subgaussian real random
  variable of unit variance.
  Then for every $\e \ge 0$ one has
  \begin{equation*}
    \P \big( s_n(A) \le \e n^{-1/2} \big)
    \le C\e +  c^n,
  \end{equation*}
  where $C > 0$  and $c \in (0,1)$.
\end{theorem*}
Recall that we have divided the unit sphere into compressible and
incompressible vectors (see Definition \ref{d: compressible} and
inequality \eqref{i: two terms}), and proved that the first term in
\eqref{i: two terms} is exponentially small. Applying Lemma \ref{l:
via distance} and \eqref{i: dist}, we reduced the estimate for the
second term to the bound for
\[
  p(\e):=\P \big ( |\pr{Z}{X_n}| \le \e \big),
\]
where $X_n$ is the $n$-th column of the matrix $A$, and $Z$ is a unit
vector orthogonal to the first $n-1$ columns. To complete the proof, we
have to show that
\begin{equation}   \label{i: scalar product}
  p(\e) \le C \e,
\end{equation}
whenever $\e \ge e^{-cn}$. Here $\pr{Z}{X_n}= \sum_{j=1}^n Z_j
\xi_j$, where $Z=(Z_1 \etc Z_n)$. Throughout the rest of the
proof set
\begin{equation} \label{i: alpha}
  \a=\b \sqrt{n},
\end{equation}
where $\b>0$ is a small absolute constant, which will be chosen at
the end of the proof.
 If $\LCD_\a(Z) \ge e^{cn}$, then \eqref{i: scalar
product} follows from Theorem \ref{t: SBP}. Therefore, our problem
has been further reduced to proving
\begin{theorem}[Random normal]                      \label{t: random normal}
  Let $X_1,\ldots,X_{n-1}$ be random vectors whose coordinates are
  independent copies of a centered subgaussian random variable
  $\xi$.
  Consider a unit vector $Z$ orthogonal to all these vectors.
  There exist constants $c, c' > 0$ such that
  $$
  \P \big( \LCD_\a(Z) < e^{cn}
     \big)
  \le e^{-c'n}.
  $$
\end{theorem}
The components of a random vector should be
arithmetically incommensurate to the extent that their essential LCD
is exponential in $n$. Intuitively, this is rather obvious for a random vector
uniformly distributed over the sphere.
It can be rigorously checked by estimating the total area of the points on the sphere, which have smaller values of the LCD.
 However,
 the distribution  of the random normal $Z$ is more involved,
  and it requires some work to confirm
this intuition.
\begin{proof}
 Let $A'$ be the $(n-1)\times n$ matrix with rows
 $X_1^T,\ldots,X_{n-1}^T$. Then $Z \in \text{Ker}(A')$.
 The matrix $A'$ has i.i.d. entries. We start with using the decomposition similar to \eqref{i: two terms}:
\begin{align*}
  &\P \big(\exists Z \in S^{n-2} \ \  \LCD_\a(Z) < e^{cn} \text{ and } A'Z=0
     \big) \\
  &\le \P \big(\exists Z \in \Comp \ \  AZ=0
        \big) \\
   & \quad + \P \big( \exists Z \in \Incomp \ \  \LCD_\a(Z) < e^{cn} \text{ and } A'Z=0
        \big).
\end{align*}
 Lemma \ref{l: compressible} implies that the first term in the right hand side does not exceed $e^{-cn}$. Formally, we have to reprove this lemma for $(n-1) \times n$ matrices, instead of the $n \times n$ ones, but the proof extends to this case without any changes.

 To bound the second term, we introduce a new decomposition of the
 sphere based on the LCD.  Recall that by Lemma \ref{l: incomp LCD},  any incompressible
 vector $a$ satisfies $\LCD_{\a}(a) \ge \l \sqrt{n}$.
 For $D >0$, set
 \[
   S_D=\{ x \in S^{n-1} \mid  D \le \LCD_\a(x) \le 2D \}.
 \]
 It is enough to prove that
 \[
   \P(\exists x \in S_D \ A'x=0) \le e^{-n}.
 \]
  whenever
 $\l \sqrt{n} \le D \le e^{cn}$. Indeed,
 the statement of the Theorem will then follow by taking the union
 bound over $D=2^k$ for $k \le cn$.

 To this end, we shall use the $\e$-net argument to bound
 $\norm{A'x}_2$ below. For a fixed $x \in S_D$, the required estimate follows
 from substituting the bound $\LCD_\a(x) \ge D$ in
 Lemma \ref{l: single}:
 \begin{equation} \label{i: single x}
  \P \big( \|A'x\|_2 < t n^{1/2} \big)
  \le ( C t )^{n-1},
 \end{equation}
 provided
  $
    t \ge \frac{(4/\pi)}{D}.
  $
  To estimate the size of the $\e$-net we use the bound for the
  essential least common denominator again.
  The simple volumetric bound is not sufficient for our purposes,
and this is the crucial step where we explore the additive structure
of $S_D$ to construct a smaller net.

\begin{lemma}[Nets of level sets]                   \label{l: net}
  There exists a $(4\a/D)$-net in $S_D$ of cardinality at most $(CD / \sqrt{n})^n$.
\end{lemma}

It is important, that the mesh of this net  depends on the small parameter $\a$, while its  cardinality is independent of it.
This feature would allow us later to use the union bound for an appropriately chosen $\a$.

The proof of Lemma \ref{l: net} is based on counting the number of integer points in a ball of a large radius.
We will show that if $x \in S_D$, then the ray $\{\l x \mid \l>0 \}$ passes within distance $\a$ from an integer point in a ball of radius $3 D$.
The number of such points is independent of $\a$, and can be bounded from the volume considerations.

\vskip 0.1in
\begin{center}
 \begin{tikzpicture}[semithick, >=stealth, scale=0.55]

   \filldraw[fill=white!90!black] (4,3) circle (6pt);

  \foreach \x in {-7,...,7}
  \foreach \y in {-7,...,7}
    \filldraw (\x,\y) circle (1pt);

  \draw (0,0) circle (6.2);

  \draw[->] (-7.5,0) -- (7.5,0);
   \draw[->] (0,-7.5) -- (0,7.5);

  \draw[] (0,0) --
    node[above, very near end, fill=white]{$\{ \lambda x \mid \lambda>0 \}$} (38:8.8);

  \draw[->] (0,0) --
    node[ fill=white]{$3 D$} (130:6.2);

 \end{tikzpicture}

\end{center}

\begin{proof}
We can assume that $4\a/D \le 1$, otherwise the conclusion is
trivial. To shorten the notation, denote for $x \in S_D$
$$
D(x) := \LCD_{\a}(x).
$$
By the definition of $S_D$, we have $D \le D(x) < 2D$. By the
definition of the essential least common denominator, there exists $p \in
\Z^n$ such that
\begin{equation}                    \label{Dx x}
  \|D(x)x - p\|_2 < \a.
\end{equation}
Therefore
$$
\Big\| x - \frac{p}{D(x)} \Big\|_2 < \frac{\a}{D(x)} \le
\frac{\a}{D} \le \frac{1}{4}.
$$
Since $\|x\|_2 = 1$, it follows that
\begin{equation}                    \label{x-pp}
  \Big\| x - \frac{p}{\|p\|_2} \Big\|_2
  < \frac{2\a}{D}.
\end{equation}

On the other hand, by \eqref{Dx x} and using  $\|x\|_2 = 1$,
$D(x) \le 2D$ and $4\a/D \le 1$, we obtain
\begin{equation}                    \label{net bounded}
  \|p\|_2 < D(x) + \a
  \le 2D + \a
  \le 3D.
\end{equation}

Inequalities \eqref{x-pp} and \eqref{net bounded} show that the set
$$
\NN := \Big\{ \frac{p}{\|p\|_2} :\; p \in \Z^n \cap B(0,3D) \Big\}
$$
is a $(2\a/D)$-net of $S_D$. Recall that, by a known volumetric
argument, the number of integer points in $B(0,3D)$ is at most
$(1+9D/\sqrt{n})^n \le (CD/\sqrt{n})^n$ (where in the last
inequality we used that by the definition of the level set, $D >
c_0\sqrt{n}$ for all incompressible vectors). Finally, we can find a
$(4\a/D)$-net of the same cardinality, which lies in $S_D$.
\end{proof}
Now we can complete the $\e$-argument. Recall that by Proposition
\ref{p: norm},
\[
  \P(s_1(A') \ge C_0 \sqrt{n} ) \le e^{-cn}.
\]
Therefore, in order to complete the proof, it is enough to show that the event
$$
\EE :=  \Big\{ \exists x \in S_D \ \  A'x =0
  \text{ and } \|A'\| \le C_0\sqrt{n} \Big\}
$$
has probability at most $e^{-n}$.

 Assume that $\EE$ occurs, and let $x \in S_D$ be such that
 $A'x =0$.
Let $\NN$ be the $(4\a/D)$-net constructed in Lemma \ref{l: net}.
Choose $y \in \NN$ such that $\norm{x-y} < 4\a/D$. Then by the
triangle inequality,
\[
  \norm{A'y}_2 \le \norm{A'}\cdot \norm{x-y}_2
  <  C_0 \sqrt{n} \cdot \frac{4 \a}{D}
  = 4 C_0 \b \frac{n}{D},
\]
if we  recall that
$\a=\b \sqrt{n}$. Set $t=4 C_0 \b \sqrt{n}/D$. Combining the estimate
\eqref{i: single x} for this $t$ with the union bound, we obtain
\begin{align*}
  \P(\EE)
  &\le \P ( \exists y \in \NN \ \norm{A'y}_2 \le t \sqrt{n})
  \le |\NN| \cdot (Ct)^{n-1}
  \le \left ( \frac{CD}{\sqrt{n}}  \right )^n \cdot (Ct)^{n-1}  \\
  &\le \left ( \frac{CD}{\sqrt{n}}  \right ) \cdot \left(4C C_0 \b \right)^{n-1}.
\end{align*}
 Since $D \le e^{c n}$, we can choose the constant $\b$ so that the right hand side of the previous inequality
 will be less than $e^{-n}$.
  The proof
of Theorem \ref{t: subgaussian} is complete.

\end{proof}

\section{Short Khinchin inequality} \label{sec: short Khinchin}

Let $1 \le p < \infty$. Recall that $\norm{\cdot}_p$ denotes the
standard $\ell_p$ norm in $\R^n$, and $B_p^n$ its unit ball.

Let $X \in \R^n$ be a  vector with independent centered random $\pm
1$ coordinates, i.e. a random vertex of the discrete cube
$\{-1,1\}^n$. The classical Khinchin inequality, Theorem \ref{t:
Khinchin}, asserts that for any $p \ge 1$ and for any vector $a \in
\R^n$, $\big ( \E |\pr{a}{X}|^p \big )^{1/p} $ is equivalent to   $\norm{a}_2$ up to multiplicative constants depending on $p$.
This equivalence can be obtained if one averages not over the whole
discrete cube, but over some small part of it. The problem how small
should this set be was around since mid-seventies.  More precisely,
\begin{quote}
  {\em Let $p \ge 1$. Find constants $\a_p,\b_p$ and a set $V \subset
  \{-1,1\}^n$ of a small cardinality such that
  \[
    \a_p \norm{a}_2
    \le  \left ( \frac{1}{|V|}
        \sum_{x \in V} |\pr{a}{x}|^p \right )^{1/p}
    \le \b_p \norm{a}_2
  \]
  for any $a \in \R^n$.
  }
\end{quote}
Deterministic constructions of sets $V$ of reasonably small
cardinality are unknown. Therefore, we shall construct the set $V$
probabilistically. Namely, we choose $N=N(n,p)$ and consider $N$
independent copies $X_1 \etc X_N$ of the random vector $X$. If $N
\ll 2^{n/2}$, in particular, if $N$ is polynomial in $n$, all vectors
$X_1 \etc X_N$ are distinct with high probability. The problem thus
is reduced to showing that with high probability, any vector $y \in
\R^n$ satisfies
  \begin{equation}  \label{short Khinchin}
    \a_p \norm{y}_2
    \le \left ( \frac{1}{N}
          \sum_{j=1}^N |\pr{y}{X_j}|^p \right )^{1/p}
    \le \b_p \norm{y}_2.
  \end{equation}
This problem can be recast in the language of random matrices. Let
$A$ be the $N \times n$ matrix with rows $X_1 \etc X_N$. Then the
inequality above means that $A$ defines a nice isomorphic embedding
of $\ell_2^n$ into $\ell_p^N$.

As in the proof of the original Khinchin inequality, we consider
cases $p=1$ and $p>2$ separately.

\subsection{Short Khinchin inequality for $p=1$}
In this case we derive the inequality \eqref{short Khinchin} in a
more general setup. Assume that the coordinates of the vector $X$
are i.i.d. centered subgaussian variables.
 The middle term in \eqref{short Khinchin} can be rewritten as $N^{-1/p} \norm{A y}_p$, where $A$ is the matrix with columns $X_1 \etc X_N$.
 In this language, establishing \eqref{short Khinchin} is equivalent to estimating the  the maximum and ther minimum of $\norm{Ay}_p$ over the unit sphere.

 Proposition \ref{p:
norm} combined with the inequality
 $ \norm{A: \ell_2^n \to \ell_1^N}
   \le  \sqrt{N} \cdot \norm{A: \ell_2^n \to \ell_2^N}
 $
 yields the following
\begin{proposition}                       \label{p: l_1 norm}
  Let $A$ be an $N \times n$ random matrix, $N \ge n$, whose entries
  are independent copies of a subgaussian random variable. Then
  $$
  \P \big( \norm{A: \ell_2^n \to \ell_1^N} > t N \big)
  \le e^{-c_0 t^2 N}
  \qquad \text{for } t \ge C_0.
  $$
\end{proposition}

This  implies the second inequality in \eqref{short Khinchin} with $p=1$ and
$\b_1=C_0$, so \eqref{short Khinchin} is reduced to the first
inequality. To establish it we apply the random matrix machinery
developed in the previous sections. Without loss of generality, we
may assume that $n \le N \le 2n$, because we are looking for small
values of $N$. Then the following Theorem will imply that the short
Khinchin inequality holds for any $N \ge n$ with  $\a_1$ depending
only on the ratio of $N/n$.
\begin{theorem}   \label{t: short Khinchin}
  Let $n,N$ be natural numbers such that $n\le N \le 2n$.
   Let $A$ be an $N \times n$ matrix, whose
  entries are i.i.d. centered subgaussian random variable
  of variance $1$. Set
  \[
    m=N-n+1.
  \]
  Then for any $\e>0$
  \[
    \P \left ( \exists x \in S^{n-1} \ \norm{Ax}_1
                  <  \e m \right )
    \le \left (\frac{CN}{m} \cdot \e \right )^{m}
       + c^n,
  \]
  where $C>0$ and $c \in (0,1)$.
\end{theorem}

\begin{proof}
Adding to the entries of $A$  small multiples of independent $N(0,1)$
variables, we may assume that the entries of $A$ are absolutely
continuous, so the matrix $A$ is of a full rank almost surely.

We start with an elementary lemma from linear algebra.
The section of $B_1^N$ by the range of the operator $A$ is a convex polytope.
The next lemma shows that the minimum of the $\ell_2$ norm of $Ay$ over the Euclidean unit sphere is attained at
a point $y$, which is mapped to a multiple of a vertex of this polytope. Such vertex will have exactly $m$ non-zero coordinates.
\begin{lemma} \label{l: extreme points}
 Let $N>n$ and let $A: \R^n \to \R^N$ be a random matrix with
 absolutely continuous entries. Let $x \in S^{n-1}$ be a vector for
 which $\norm{Ax}_1$ attains the minimal value. Then
 \[
   |\text{ \rm supp} (Ax)| = N-n+1
 \]
 almost surely.
\end{lemma}

\begin{proof}
Let $E=A \R^n$ and let $K=B_1^N \cap E$.  Set $y=Ax/\norm{Ax}_1$.
Since the function $g: S^{n-1} \to (0,\infty), \ g(u)=\norm{Au}_1$ attains the minimum at $u=x$, the function
$f: K \to (0, \infty), \ f(z)=\norm{A^{-1}|_E \, z}_2$ attains the
maximum over $K$ at $z=y$.
 The convexity of $\norm{\cdot}_2$ implies that
$y$ is an extreme point of $K$. Since $K$ is the intersection of
the octahedron $B_1^N$ with an $n$-dimensional subspace, this means
that $|\text{supp } y| \le N-n+1$. Finally, since the entries of $A$
are absolutely continuous, any coordinate subspace $F \subset \R^N$,
whose dimension does not exceed $N-n$, satisfies $E \cap F = \{0\}$
a.s. Therefore, $|\text{supp } y| = N-n+1$.
\end{proof}
This lemma allows us to reduce the minimum of $\norm{Ax}_1$ over the
whole sphere $S^{n-1}$ to a certain finite subset of it.
Indeed, to each
subset $J \subset \{1 \etc N\}$ of cardinality $m=N-n+1$ corresponds
a unique pair of extreme points $v_J$ and $-v_J$ of $K$ such that
$\sum_{j \in J} |v_J(j)|=1$ and $v_J(j)=0$ whenever $j \notin J$.
Let $A_{J'}$ be the matrix consisting of the rows of $A$, whose
indices belong to $J'=\{1 \etc N \} \setminus J$.
  The vector $y_J \in S^{n-1}$ such that $Ay_J=t v_J$ for
some $t>0$ is uniquely defined by the matrix $A_{J'}$ via the condition $A_{J'}y_J=0$. By Lemma
\ref{l: extreme points},
\[
  \min \{ \norm{Ay}_1 \mid y \in S^{n-1} \}
  = \min \{ \norm{Ay_J}_1 \mid J \subset \{1 \etc N\}, \
                              |J|=m \}.
\]
To finish the proof, we estimate  $\norm{Ay_J}_1$ below and apply
the union bound over the sets $J$. Fix a set $J \subset \{1 \etc N\}$ of cardinality
$m$. Denote the rows of the matrix $A_{J'}$ by $X_1^T \etc X_{n-1}^T$.
The condition $A_{J'}y_J=0$ means that $y_J$ is orthogonal to each of these vectors.
Applying Theorem \ref{t: random normal} to the vectors $X_1 \etc X_{n-1}$, we
conclude that
 \begin{equation}                  \label{exp LCD}
  \P \big( \LCD_\a(y_J) < e^{cn}
     \big)
  \le e^{-c'n}.
 \end{equation}
  Conditioning  on the matrix $A_{J'}$, we may regard the vector
  $y_J$ as fixed.
Denote a row of the matrix $A_J$ by $Y^T$, so the coordinates of $A_J
y_J$ are distributed like $\pr{Y}{y_J}$. If $\LCD_\a(y_J) \ge
e^{cn}$, then by Theorem \ref{t: SBP}
\[
  \P (|\pr{Y}{y_J}| \le \e \mid A_{J'} ) \le C \e,
\]
whenever $\e> C e^{-c n}$. Then taking expectation over $A_{J'}$ and
using \eqref{exp LCD} yields
\[
  \P (|\pr{Y}{y_J}| \le \e  ) \le C \e +C e^{-c n}+e^{-c'n}
\]
for any $\e> 0$. Coordinates $\zeta_j, \ j \in J$ of the vector
$A_Jy_J$ are i.i.d. random variables.
 Tensorization Lemma \ref{l: tensorization}
can be easily reproved for $\sum |\zeta_j|$ instead of $\sum
\zeta_j^2$. In this form it implies
\[
  \P( \norm{A y_J}_1 \le  \e  m)
  =  \P( \norm{A_J y_J}_1 \le \e  m)
  \le \left(C \e +C e^{-c n} \right)^m
\]
for any $\e> 0$. Finally, taking the union bound over all sets $J$,
we obtain
\begin{align*}
  \P( \exists J \ |J|=m, \ \norm{A y_J}_1 \le  \e  m)
  &\le \binom{N}{m} \cdot \left(C \e +C e^{-c n} \right)^m  \\
  &\le \left ( \frac{CN}{m} \cdot \e \right )^m +C e^{-c'' n}. \qedhere
\end{align*}
\end{proof}

Assume now that $N$ is in a fixed proportion to $n$, and define $\d$ by
$N=(1+\d)n$. In this notation, Theorem \ref{t: short Khinchin} reads
\[
 \P \left( \exists \, x \in S^{n-1} \ \norm{Ax}_1< \e \d n \right)
 \le \left( \frac{C \e}{ \d} \right)^{\d n+1} +c^n.
\]
Set $\e = c' \d$, where the constant $c'$  is chosen to make the right hand side of the inequality above smaller than $1$.
Then the previous estimate shows that, with
high probability, the short Khinchin inequality holds for $N=(1+\d)n$
independent subgaussian vectors $X_1 \etc X_N$ with $p=1$ and constants $\a_1=c \d^2, \ \b_1=C_0$:
\[
   \forall \, y \in \R^n \quad
    c \d^2 \norm{y}_2
    \le \frac{1}{N}
          \sum_{j=1}^N |\pr{y}{X_j}|
    \le C_0\norm{y}_2.
\]
 Theorem \ref{t: short Khinchin} proves more than the short Khinchin
 inequality. Combining it with Proposition \ref{p: norm}, we show
 that
 \begin{equation}  \label{1-2 norms}
   \forall x \in \R^n \quad
   \e \d n \norm{x}_2 \le \norm{Ax}_1 \le \sqrt{N}\norm{Ax}_2
   \le C' n \norm{x}_2
\end{equation}
with probability greater than $1-C\exp(-cn)-(\e/\bar{c}\d)^{\d n}$.
The second inequality here follows from Cauchy--Schwarz, and the third one from Proposition  \ref{p: norm}.
Inequality \eqref{1-2 norms}  immediately yields a lower bound for the smallest singular
value of a rectangular random matrix.
\begin{corollary}
  Let $n,N,\d, A,\e$ be as above. Then the smallest singular value of
  $A$ is bounded below by $\e \d\cdot \sqrt{n}$ with probability at least
  $1-c^n-(\e /\bar{c}\d)^{\d n}$.
\end{corollary}
This bound is not sharp for small $\d$. The optimal estimate
\[
      \P \Big( s_n(A) \le \e \big (\sqrt{N} - \sqrt{n-1} \big ) \Big)
    \le (C \e)^{N-n+1} + c^N,
\]
valid
for all $n,N$ and $\e$, was  obtained in
\cite{RV-rectangular}.

Another application of the inequality \eqref{1-2 norms} is a bound on the diameter of a random section of the octahedron $B_1^N$.
A celebrated theorem of Kashin \cite{Ka} states that a random $n$-dimensional
section of the standard octahedron $B_1^N$ of dimension $N = \lfloor (1+\d) n \rfloor$ is
close to the section of the inscribed  ball $(1/\sqrt{N}) B_2^N$.
The optimal estimates for the diameter of a random section of the
octahedron were obtained by Garnaev and Gluskin \cite{GG}. Recently
the attention was attracted to the question whether the almost
spherical sections of the octahedron can be generated by simple
random matrices, in particular by a random $\pm 1$ matrix. A general
result proved in \cite{LPRTV} implies that if $N = \lfloor (1+\d) n \rfloor$ with $\d
\ge c /\log n$, then a random $N \times n$ matrix $A$ with independent
subgaussian entries generates a section of the octahedron $B_1^N$
which is not far from the ball with probability exponentially close
to 1.
More precisely, if $E=A \R^n \subset \R^N$, then
\[
 (1/\sqrt{N}) B_2^N \cap E \subset B_1^N \cap E \subset \varphi (\d) \cdot (1/\sqrt{N}) B_2^N,
\]
where $\varphi(\d) \le C^{1/\d}$.

For random $\pm 1$ matrices this result was improved by
Artstein-Avidan at al. \cite{AFMS}, who proved a polynomial type
estimate for the diameter of a section $\varphi(\d) \le (1/\d)^{\a}$ for $\a>5/2$ and $\d \ge C n^{-1/10}$.
Using \eqref{1-2 norms} we obtain a polynomial estimate for the
diameter of sections for smaller values of $\d$.
\begin{corollary}
  Let $n,N$ be natural numbers such that $n<N<2n$. Denote $\d=(N-n)/n$.
   Let $\xi$ be a centered subgaussian random variable.
    Let $A$ be an $N \times n$ matrix, whose
  entries are independent copies of $\xi$ and let $E=A \R^n$.
  Then for any $\e>0$
  \[
    \P \left ( \frac{1}{\sqrt{N}} B_2^N \cap E \subset B_1^N \cap E \subset \frac{c}{\e \d} \cdot  \frac{1}{\sqrt{N}} B_2^N \right )
    \ge 1- c^n - (\e/\bar{c}\d)^{\d n}.
  \]
\end{corollary}
 Note that to make the probability bound non-trivial, we have to assume that $\e = c' \d$ for some $0<c'<\bar{c}$.
 In this case the corollary means that a random $n$-dimensional subspace $E$ satisfies
 \[
   \frac{1}{\sqrt{N}} B_2^N \cap E \subset B_1^N \cap E
   \subset \left(\frac{c}{\d^2} \right) \cdot  \frac{1}{\sqrt{N}} B_2^N.
 \]
 This inclusion remains non-trivial as long as $\left(\frac{c}{\d^2} \right)< \sqrt{N}$, i.e., as long as $\d>c N^{-1/4}$.

\subsection{Short Khinchin inequality for $p>2$} \label{ss: p>2}

The case $p>2$ requires a completely different approach. In this
case we will prove the short Khinchin inequality without the assumption that the coordinates of the random vector $X$
are independent. We shall assume instead that
$X$ is isotropic and subgaussian. The first property means that for
any $y \in S^{n-1}$
\[
  \E \pr{X}{y}^2 = 1,
\]
while the second means that for any $y \in S^{n-1}$ the random variable
$\pr{X}{y}$ is centered subgaussian. By Theorem \ref{t: Hoeffding}, any
random vector with independent centered subgaussian coordinates of variance 1 is
isotropic subgaussian. This includes, in particular, an
appropriately scaled random vertex of the discrete cube
$\{-1,1\}^n$.

 We prove the following Theorem \cite{GR}.
\begin{theorem}             \label{t: p>2}
  Let $X$ be an isotropic subgaussian vector in $\R^n$. Let $X_1
  \etc X_N$ be independent copies of $X$. Let $p>2$ and
  $N \ge n^{p/2}$. Then, with high probability,  the inequalites
  \[
    c \norm{y}_2
    \le \left ( \frac{1}{N} \sum_{j=1}^N |\pr{y}{X_j}|^p
         \right )^{1/p}
    \le C \sqrt{p} \norm{y}_2
  \]
  hold  for all $y \in \R^n$.
\end{theorem}

\begin{proof}
As in the classical Khinchin inequality, the first inequality in
Theorem~\ref{t: p>2} is easy. Denote, as before, by $A$ the $N
\times n$ matrix with rows $X_1 \etc X_N$.
 Assume that $n$ is large enough, so
that $N \ge n^{p/2} \ge \d_0^{-1} n$, where $\d_0$ is the constant
from Proposition \ref{p: rectangular}.
By the remark after this proposition, it is applicable to the matrix $A$
despite the fact that its entries are dependent.
Combining Proposition  \ref{p: rectangular}
with the inequality $\norm{y}_2 \le N^{1/2-1/p} \cdot \norm{y}_p$,
valid for all $y \in \R^N$, we obtain
\[
    \P \big( \min_{x \in S^{n-1}} \norm{Ax}_p \le c_1 N^{1/p}
         \big)
    \le e^{-c_2 N},
\]
which establishes the left inequality with probability exponentially
close to $1$.

If the vectors $X_1 \etc X_N$ were independent standard gaussian,
then the right inequality in Theorem \ref{t: p>2} would follow from the classical Gordon--Chevet inequality
for the norm of the Gaussian linear operator, see e.g., \cite{DS}.
We will establish an analog of this inequality for isotropic subgaussian vectors.
To this end, we use the method of majorizing
measures, or generic chaining, developed by Talagrand \cite{Ta}. Let
$\{X_t\}_{t \in T}$ be a real-valued random process, i.e., a collection of
interdependent random variables, indexed by some set $T$. In the
setup below, we can assume that $T$ is finite or countable,
eliminating the question of measurability of $\sup_{t \in T} X_t$.
We shall call the process $\{X_t\}_{t \in T}$ centered if $\E X_t=0$
for all $t \in T$.
\begin{definition}
 Let $(T,d)$ be a metric space. A random process $\{X_t\}_{t \in T}$
 is called subgaussian with respect to the metric $d$ if for any
 $t,s \in T, \ t \neq s$ the random variable $(X_t-X_s)/d(t,s)$ is
 subgaussian.
 A random process $\{G_t\}_{t \in T}$
 is called Gaussian with respect to the metric $d$ if for any
 finite set $F \subset T$ the joint distribution of $\{G_t\}_{t \in F}$
 is Gaussian, and for any
 $t,s \in T, \ t \neq s$  $(G_t-G_s)/d(t,s)$ is
 $N(0,1)$ random variable.
\end{definition}
We use a fundamental result of Talagrand \cite{Ta} comparing subgaussian and Gaussian processes.
\begin{theorem}[Majorizing Measure Theorem]
    \label{t: Majorizing Measure}
  Let $(T,d)$ be a metric space, and let $\{G_t\}_{t \in T}$ be a
  Gaussian random process with respect to the metric $d$. For any centered random
  process $\{X_t\}_{t \in T}$, which is subgaussian with respect to
  the same metric,
  \[
    \E \sup_{t \in T} X_t \le C \, \E \sup_{t \in T} G_t.
  \]
\end{theorem}

For $(s,y) \in \R^N \times \R^n$ define the random variable
$X_{s,y}$ by
\[
  X_{s,y}= \sum_{j=1}^N s_j \pr{X_j}{y}.
\]
Let us show that for any $T \subset B_2^N \times B_2^n$, the random process $\{
X_{s,y} \}_{(s,y) \in T}$ is subgaussian with respect to the
Euclidean metric. For any $(s,y), (s',y') \in T$,
\[
  X_{s,y}-X_{s',y'}
  = \sum_{j=1}^N \Big ( (s_j-s_j') \pr{X_j}{y} +s_j' \pr{X_j}{y-y'}
                  \Big ).
\]
Let  $\l \in \R$. Since the vector $X$ is centered subgaussian, for
any $z \in \R^N$  $\exp ( \l \pr{X}{z}) \le \exp (C \l^2
\norm{z}_2^2)$. Hence, using independence of $X_j$ and applying
Cauchy--Schwartz inequality, we get
\begin{align*}
  \E &\exp \big (\l (X_{s,y}-X_{s',y'} ) \big )  \\
  &= \prod_{j=1}^N \E \Big [ \exp \big (\l (s_j-s_j') \pr{X_j}{y} \big )
     \, \cdot
     \exp \big (\l s_j' \pr{X_j}{y-y'} \big ) \Big ]  \\
  &\le \prod_{j=1}^N \exp \big (2 C \l^2 ((s_j-s_j')^2 \norm{y}_2^2) \big )
     \, \cdot
     \prod_{j=1}^N \exp \big (2 C \l^2 (s_j'^2 \norm{y-y'}_2^2) \big )  \\
  &\le \exp \big ( 2C \l^2 (\norm{s-s'}_2^2 + \norm{y-y'}_2^2 ) \big ).
\end{align*}
 The last inequality follows because $(s,y), (s',y') \in T \subset B_2^N \times B_2^n$.
 By Theorem \ref{t: def subgaussian} this means that the random
 variable
 \[
   \frac{X_{s,y}-X_{s',y'}}{\norm{(s,y)-(s',y')}_2}
 \]
  is
 subgaussian, so the process $(X_{s,y})_{(s,y) \in T}$ is subgaussian with respect to the $\ell_2$ metric.

Now we will consider a Gaussian process with respect to the same metric.
Let $Y$ and $Z$ be independent standard Gaussian vectors in $\R^n$ and $\R^N$
respectively. Set
\[
  G_{s,y} = \pr{s}{Z} + \pr{y}{Y}.
\]
Then for any $T \subset \R^N \times \R^n$, $\{G_{s,y}\}_{(s,y) \in
T}$ is a Gaussian process with respect to the Euclidean metric. Let
$1/p+1/p^*=1$, and set $T= B_{p^*}^N \times B_2^n \subset B_2^N
\times B_2^n$. By the Majorizing Measure Theorem
  \[
    \E \sup_{(s,y) \in T} X_{s,y}
    \le C \, \E \sup_{(s,y) \in T} G_{s,y}.
  \]
Therefore, writing the $\ell_p$ norm as the supremum of the values of functionals over the unit ball of the dual space, we obtain
\begin{align*}
  \E \sup_{y \in B_2^n} \left ( \frac{1}{N}
       \sum_{j=1}^N |\pr{X_j}{y}|^p \right )^{1/p}
  &= \frac{1}{N^{1/p}} \E \sup_{s \in B_{p^*}^N} \sup_{y \in B_2^n}
          \sum_{j=1}^N s_j \pr{X_j}{y}   \\
  &\le \frac{C}{N^{1/p}} \E \sup_{s \in B_{p^*}^N} \sup_{y \in B_2^n}
          G_{s,y}
  = \frac{C}{N^{1/p}} \big (\E \norm{Z}_p + \E \norm{Y}_2 \big ) \\
  &\le C \left ( \sqrt{p} + \frac{\sqrt{n}}{N^{1/p}} \right).
\end{align*}
Since $N \ge n^{p/2}$, the last expression does not exceed $C'
\sqrt{p}$. To complete the proof we combine this estimate of the
expectation with Chebyshev's inequality.
\end{proof}

\begin{remark}
  The same proof can be repeated for an general normed space, instead of the space $\ell_p$. This would establish a version of Gordon--Chevet inequality valid for a general isotropic subgaussian vector. We omit the details.
\end{remark}

Note that Theorem \ref{t: p>2} implies that the matrix $A$ formed by
the vectors $X_1 \etc X_N$ defines a subspace of $\ell_p^N$ which is
close to Euclidean, so Theorem \ref{t: p>2} can be viewed as an analog of
the Isomorphic Dvoretzky's Theorem of Milman and Schechtman \cite{MSc}.
 This, in particular, means that the bound $N \ge
n^{p/2}$ is optimal (see  e.g., \cite{GM} for details).

\section{Random unitary and orthogonal perturbations} \label{sec: unitary}

 The need for probabilistic bounds for the smallest singular value of a random matrix from a certain class arises in many intrinsic problems of the random matrix theory. Such bounds are the standard step in many proofs based on the convergence of Stieltjes transforms of the empirical measures to the Stieltjes transform of the limit measure. One of the examples, where such bounds become necessary is the Circular Law \cite{GT, TV circ1, TV circ2}. The proof of this law requires the lower bound on the smallest singular value of a random matrix with i.i.d. entries, which was obtained above.
  Another setup, where such bounds become necessary, is provided by the Single Ring Theorem of Guionnet, Krishnapur and Zeitouni \cite{GKZ}. The proof of this theorem deals with another natural class of random matrices, namely random unitary or orthogonal perturbations of a fixed  matrix.

  Let us consider the complex case first. Let $D$ be a fixed $n \times n$ matrix, and let $U$ be a random matrix uniformly distributed over the unitary group $U(n)$. In this case the solution of the qualitative invertibility problem is trivial, since the matrix $D+U$ is non-singular with probability $1$. This can be easily concluded by considering the determinant of $D+U$. The determinant, however, provides a poor tool for studying the quantitative invertibility problem. In regard to this problem we will prove the following theorem.
  \begin{theorem} \label{th: unitary}
   Let $D$ be  an arbitrary $n \times n$ matrix, $n \ge 2$. Let $U$ be a random matrix uniformly distributed over the unitary group $U(n)$. Then
   \[
    \P( \smin(D+U) \le t) \le t^c n^C \quad \text{for all } t>0.
   \]
   Here $C$ and $c$ are absolute constants.
  \end{theorem}

  An important feature of Theorem \ref{th: unitary} is its independence of the matrix $D$. This independence is essential for the Single Ring Theorem.

  The statement similar to Theorem \ref{th: unitary} fails in the real case, i.e.,  for random matrices distributed over the orthogonal group. Indeed, suppose that $n$ is odd. If $-D, U \in SO(n)$, then $-D^{-1} U \in SO(n)$ has the eigenvalue $1$, and the matrix $D+U= D(D^{-1}U+ I_n)$ is singular. Therefore, if $U$ is uniformly distributed over $O(n)$, then $\smin(D+U)=0$ with probability at least $1/2$. Nevertheless, it turns out that this is essentially the only obstacle to the extension of Theorem \ref{th: unitary} to the orthogonal case.
\begin{theorem}[Orthogonal perturbations]					\label{th: orthogonal}
  Let $D$ be a fixed $n \times n$ real matrix, $n \ge 2$. Assume that
  \begin{equation}				\label{eq: D}
  \|D\| \le K, \quad \inf_{V \in O(n)} \|D-V\| \ge \d
  \end{equation}
  for some $K \ge 1$, $\d \in (0,1)$.
  Let $U$ be a random matrix uniformly distributed over the orthogonal group $O(n)$.
  Then
  $$
  \P(\smin(D+U) \le t) \le t^c (Kn/\d)^C, \quad t > 0.
  $$
\end{theorem}
Similarly to the complex case, this bound is uniform over all matrices $D$ satisfying \eqref{eq: D}. This condition is relatively mild: in the case when $K=n^{C_1}$ and $\d=n^{-C_2}$ for some constants $C_1,C_2>0$, we have
  $$
  \P(\smin(D+U) \le t) \le t^c n^C, \quad t > 0,
  $$
as in the complex case.
It is possible that the condition $\|D\| \le K$ can be eliminated from the Theorem~\ref{th: orthogonal}.
However, this is not crucial because such condition
already appears in the Single Ring Theorem.

   The problems we face in the proofs of Theorems \ref{th: unitary} and \ref{th: orthogonal} are significantly different from those appearing in Sections \ref{sec: continuous}, \ref{sec: random normal}. In the case of the independent entries the argument was based on the analysis of the small ball probability $\P(\norm{Ax}_2<t)$ or $\P(\norm{Ax}_1)<t$ for a fixed vector $x$. As shown in Section \ref{sec: small ball}, the decay of this probability as $t \to 0$ is determined by the arithmetic structure of the coordinates of $x$. In contrast to this, the arithmetic structure plays no role in Theorems \ref{th: unitary} and \ref{th: orthogonal}. The difficulty lies elsewhere, namely in the lack of independence of the entries of the matrix. We will have to introduce a set of the independent random variables artificially. These variables have to be chosen in a way that allows one to express  tractably  the smallest singular value in terms of them.
    To illustrate this approach, we present the proof of Theorem \ref{th: unitary} below. The proof of Theorem \ref{th: orthogonal} starts with the similar ideas, but requires new and significantly more delicate arguments. We refer the reader to \cite{RV uni-ort} for the details.

\begin{proof}[Proof of Theorem \ref{th: unitary}]
   Throughout the proof we fix $t>0$ and introduce several small  and large  parameters depending on $t$. The values of such parameters will be chosen of orders $t^a$, where $0<a<1$ for the small parameters, and $t^{-b}$, $0<b<1$ for the large ones. This would allow us to introduce an hierarchy of parameters, and disregard the terms corresponding to the smaller ones. Also, note that we have to prove Theorem \ref{th: unitary} only for $t< n^{-C'}$ for a given constant $C'$, because for larger values of $t$ its statement can be made vacuous by choosing a large constant $C$. This observation would allow us to use bounds of the type $\sqrt{n} t^a \le t^{a'}$ whenever $a<a'$ are constants.

   {\small \sffamily
   For convenience of a reader, we include a special paragraph entitled ``Choice of the parameters'' in the analysis of each case. In these paragraphs we list the constraints that the small and large parameters must satisfy, as well as the admissible numerical values of those parameters. These paragraphs will be printed in sans-serif and can be omitted on the first reading.
   }

   To simplify the argument, we will also assume that $\norm{D} \le K$, as in Theorem \ref{th: orthogonal}. The proof of Theorem \ref{th: unitary} without this assumption can be found in \cite{RV uni-ort}.

   \subsection{Decomposition of the sphere and introduction of local and global perturbations}
   We have to bound $\smin(U+D)$, which is the minimum of $\norm{(D+U)x}_2$ over the unit sphere. For every $x \in S^{n-1}$,  there is a coordinate $x_j$ with $|x_j| \ge 1/\sqrt{n}$. Hence, the union bound yields
   \[
    \P(\smin(D+U) \le t) \le \sum_{j=1}^n \P\left( \inf_{x \in S_j} \norm{(U+D)x}_2 \le t \right),
   \]
   where
   \[
    S_j= \left \{ x \in S^{n-1} \mid |x_j| \ge 1/\sqrt{n} \right \}.
   \]
   All terms on the right hand side of the inequality above can be estimated in the same way. So, without loss of generality we will consider the case $j=1$. Note that the application of the crude union bound here may have increased the probability estimate of Theorem \ref{th: unitary} $n$ times. This, however, is unimportant, since we allow the coefficient $n^C$ anyway.

   The proof of the theorem reduces to the estimate of
   \begin{equation}\label{eq: set S_1}
     \P \left( \inf_{x \in S_1} \norm{(U+D)x}_2 \le t \right).
   \end{equation}
   The structure of the set $S_1$ gives a special role to the first coordinate. This will be reflected in our choice of independent random variables. If $R,W \in U(n)$ are any matrices, and $V$ is uniformly distributed over $U(n)$, then the matrix $U=V^{-1} R^{-1}W$ is uniformly distributed over $U(n)$ as well. Hence, if we assume that the matrices $R$ and $W$ are random and independent of $V$, then this property would remain valid for $U$. The choice of the distributions of $R$ and $W$ is in our hands. Set
   \[
    R= \text{diag}(r, 1 \etc 1),
   \]
   where $r$ is a random variable uniformly distributed over $\{ z \in \mathbb{C} \mid |z|=1 \}$. This is a ``global'' perturbation, since we will need the values of $r$, which are far from $1$. The matrix $W$ will be ``local'', i.e., it will be a small perturbation of the identity matrix. Let $\e>0$ be a ``small'' parameter, and set $W=\exp(\e S)$,
      where $S$ is an $n \times n$ skew-symmetric matrix, i.e. $S^*=-S$. Although the matrix $W$ is  unitary,  the dependence of its entries on the entries of $S$ is hard to trace. To simplify the structure, we consider the linearization of  $W$,
   \[
    W_0=I+\e S.
   \]
   The matrix $W_0$ is not unitary, but its distance to the group $U(n)$ is at most $\norm{W-W_0} \le \e^2 \norm{S}^2$.
    Thus, for any $x \in S_1$,
   \begin{align*}
    \|(D+U)x\|_2
  &= \|(D + V^{-1} R^{-1} W) x\|_2
  = \|(RVD + W)x\|_2   \\
  &\ge \|(RVD + W_0)x\|_2 - \norm{W-W_0}  \\
  &\ge \|(RVD + I + \e S)x\|_2 -  \e^2 \|S\|^2.
   \end{align*}
   We will use $S$ to introduce a collection of independent random variables.
   Set
   \begin{equation}				\label{eq: S}
S =
\begin{bmatrix}
  \sqrt{-1}\, s & -Z^T \\
  Z & 0
\end{bmatrix}
\end{equation}
where $s \sim N_{\R}(0,1)$ and $Z \sim N_{\R}(0, I_{n-1})$ are independent real-valued standard normal random
variable and vector respectively. Clearly, $S$ is skew-Hermitian. If $K_0$ is a ``large'' parameter, $K_0=t^{-b_0}$, then by Proposition \ref{p: norm},  \[
  \P(\norm{Z}_2 \ge K_0 \sqrt{n}) \le \exp (-c_0 K_0^2 n) \le t
 \]
 for all sufficiently small $t>0$.
  This means that $\norm{S}^2 \le K_0^2 n$ with probability  close to $1$. Disregarding an event of a small probability, we reduce the problem to obtaining a lower bound for
\[
  \inf_{x \in S_1} \norm{(RVD + I + \e S)x}_2,
\]
provided that the bound we obtain is of order at least $\e$. Indeed, we may assume that $K_0^2 n \e^2 \ll \e$, if $\e$ is chosen small enough.
   \vskip 0.05in
   {\small \sffamily
   {\bf Choice of the parameters}. The second order term  $ 2 \e^2 \|S^2\|$ should not affect the estimate of $\P( \inf_{x \in S_1} \norm{Ax} \le t)$. To guarantee it,  we require that
   \[
     K_0^2 n \e^2 \le t/2.
   \]
   Also, to bound the probability by a power of $t$, we have to assume that
   \[
     \exp (-c_0 K_0^2 n) \le t^c
     \]
     for some $c>0$. Both inequalities are satisfied for small $t$ if $\e=t^{0.6}$ and $K_0=t^{-0.05}$.
   }
   \vskip 0.05in

  Starting from this moment we will condition on the matrix $V$ and evaluate the conditional probability with respect to the random matrices $R$ and $S$. The original random structure will be lost after this conditioning. However, we introduced a new independent structure in the form of the matrices $R$ and $S$, and it will be easier to manipulate.
  Each of the matrices $R$ and $S$ alone is insufficient to obtain any meaningful estimate.  Nevertheless, the combination of these two sources of randomness, a local perturbation $S$ and a global perturbation $R$, produces enough power to conclude that $RVD + I + \e S$ is typically well invertible, and this leads to the proof of Theorem \ref{th: unitary}.

  Summarizing the previous argument, we conclude that our goal is to bound
  \[
   \P(\inf_{x \in S_1} \norm{Ax}_2 \le t),
  \]
  where
  \begin{equation} \label{eq: A}
    A = RVD + I + \e S
     =:
    \begin{bmatrix}
      A_{11} & Y^T \\
      X & B^T
   \end{bmatrix},
  \end{equation}
  $X,Y \in \C^{n-1}, \ B$ is an $(n-1) \times (n-1)$ matrix,
  and $\e=t^a$. Here we decomposed the matrix $A$ separating the first coordinate to emphasize its special role.
  For  future reference we write $A$ in terms of the components of the matrix $VD$, and random variables $r, s$, and $Z$ exposing the dependence on these random parameters:
  \begin{equation} \label{eq: A precise}
    A  =
   \begin{bmatrix}
      A_{11} & Y^T \\
      X & B^T
   \end{bmatrix}
   =
    \begin{bmatrix}
       ra + 1 + \sqrt{-1}\, \e s & (r v - \e Z)^T \\
       u + \e Z & B^T
    \end{bmatrix}.
  \end{equation}
  Here $a \in \C$, $u,v \in \C^{n-1}$, and the matrix $B$ are independent of $r, s$, and $Z$.
  After conditioning on $V$, we can treat them as constants.

   The further strategy takes into account the properties of the matrix $B$.
  Depending on the invertibility properties of this matrix, we condition on some of the random variables $r,s$, and $Z$, and use the other ones to show that $A$ is well-invertible with high probability.

  \subsection{Case 1: $B$ is poorly invertible}
  Assume that $\smin(B) \le \l_1 \e$, where $\l_1$ is another ``small'' parameter ($\l_1=t^{a_1}$ for $0<a_1<1$). In this case we will condition on $r$ and $s$, and rely on $Z$ to obtain the probability bound. We know that there exists a vector $\tilde{w} \in S^{n-2}$ such that $\norm{B \tilde{w}}_2 \le \l_1 \e$.
  Let $x \in S_1$ be arbitrary. We can express it as
\begin{equation*}				
x =
\begin{bmatrix}
x_1 \\ \tilde{x}
\end{bmatrix},
\quad \text{where } |x_1| \ge \frac{1}{\sqrt{n}}.
\end{equation*}
Set
$$
w =
\begin{bmatrix}
0 \\ \tilde{w}
\end{bmatrix} \in \C^n.
$$
Using the decomposition of $A$ given in \eqref{eq: A}, we obtain
\begin{align*}
\|Ax\|_2
&\ge |w^T A x| =
  \left| \begin{bmatrix}
  0 & \tilde{w}^T
  \end{bmatrix}
  \begin{bmatrix}
  A_{11} & Y^T \\
  X & B^T
  \end{bmatrix}
  \begin{bmatrix}
  x_1 \\ \tilde{x}
  \end{bmatrix} \right| \\
&= |x_1 \cdot \tilde{w}^T X + \tilde{w}^T B^T \tilde{x}| \\
&\ge |x_1| \cdot |\tilde{w}^T X| - \|B \tilde{w}\|_2 		\quad \text{(by the triangle inequality)} \\
&\ge \frac{1}{\sqrt{n}} \, |\tilde{w}^T X| - \l_1 \e	
		\quad \text{(using $|x_1| \ge 1/\sqrt{n}$).}
\end{align*}
 By the representation \eqref{eq: A precise}, $X = u + \e Z$, where $u \in \C^{n-1}$ is a vector independent of $Z$. Taking the infimum over
$x \in S_1$, we obtain
$$
\inf_{x \in S_1} \|Ax\|_2
\ge \frac{1}{\sqrt{n}} \, |\tilde{w}^T u + \e \tilde{w}^T Z| - \l_1 \e.
$$
Recall that $\tilde{w}$, $u$ are fixed vectors, $\|\tilde{w}\|_2 = 1$, and $Z \sim N_\R(0,I_{n-1})$.
Then $ \tilde{w}^T Z=  \gamma$ is a complex normal random variable of variance $1$: $\E|\g|^2=1$.
 This means that $\E \big(\text{Re}(\g)\big)^2 \ge 1/2$ or $\E \big(\text{Im}(\g)\big)^2 \ge 1/2$.
 A quick density calculation yields
the following bound on the conditional probability:
$$
\P_Z \left\{ |\tilde{w}^T u + \e \tilde{w}^T Z| \le 2 \l_1 \e \sqrt{n} \right\}
\le C \l_1 \sqrt{n}.
$$
Therefore, a similar bound holds unconditionally.
Thus, combining the previous estimates, we conclude that in case when $\smin(B) \le \l_1 \e$, and if $\e$ and $\l_1$ are chosen so that $\l_1 \e \ge t$, we have
\begin{align*}
 \P( \inf_{x \in S_1} \|Ax\|_2 \le t   )
& \le \P( \frac{1}{\sqrt{n}} \, |\tilde{w}^T X| - \l_1 \e	\le t) \\
& \P \left\{ |\tilde{w}^T u + \e \tilde{w}^T Z| \le 2 \l_1 \e \sqrt{n} \right\}
\le C \l_1 \sqrt{n} =C \sqrt{n} \cdot t^{a_1}.
\end{align*}
   \vskip 0.05in
   {\small \sffamily
   {\bf Choice of the parameters}. The constraint
   \[
     \l_1 \e \ge t,
   \]
   appearing in this case, holds if we take $\l_1=t^{0.1}$.
   }
   \vskip 0.05in

\subsection{Case 2: $B$ is nicely invertible}
 Assume that $\smin(B) \ge \l_2$, where $\l_2=t^{a_2}$ is a ``small'' parameter.
 In this case, we will also use only the local perturbation, however the crucial random variable will be different. We will condition on $r$ and $Z$, and use the dependence on $s$ to derive the conclusion of the theorem.

  Set
 \[
   M=
   \begin{bmatrix}
      1 & 0 \\ 0 & (B^T)^{-1}
   \end{bmatrix},
 \]
 then $\norm{M} \le \l_2^{-1}$. Therefore,
 \[
  \inf_{x \in S_1} \norm{Ax}_2 \ge \l_2 \inf_{x \in S_1} \norm{M Ax}_2.
 \]
 The matrix $M A$ has the following block representation:
 \[
   M A=
   \begin{bmatrix}
      A_{11} & Y^T \\  (B^T)^{-1}X & I_{n-1}
   \end{bmatrix}.
 \]
 Recall that we assumed that $\norm{D} \le K$ where $K$ is a constant. Combining this with the already used inequality $\norm{Z}_2 \le K_0 \sqrt{n}$, which holds outside of the event of exponentially small probability, we conclude that
 $Y= r v - \e Z$ satisfies
 \[
   \norm{Y}_2 \le \norm{v}_2+ \e \norm{Z}_2 \le 2K
 \]
   if $\e  K_0 \sqrt{n} \le K$. To bound $\inf_{x \in S_1} \norm{Ax}_2$, we use an observation that
 \[
   \begin{bmatrix}
      1 & -Y^T
   \end{bmatrix}
   \cdot
   \begin{bmatrix}
      Y^T \\   I_{n-1}
   \end{bmatrix}
   =0.
 \]
 This implies that for every $x \in S_1$,
 \begin{align*}
   \norm{M Ax}_2
   &\ge \frac{1}{\norm{[1 \  -Y^T]}_2} \cdot
   \left|
   \begin{bmatrix}
      1 & -Y^T
   \end{bmatrix}
   M A
   \begin{bmatrix}
      x_1 \\ \tilde{x}
   \end{bmatrix}
   \right| \\
   &\ge \frac{1}{2 K} \cdot |A_{11}-Y^T (B^T)^{-1}X| \cdot |x_1| \\
   &\ge \frac{1}{2 K \sqrt{n}} \cdot |A_{11}-Y^T (B^T)^{-1}X|.
 \end{align*}
 The right hand side of this inequality does not depend on $x$, so we can take the infimum over $x \in S_1$ in the left hand side.
 Combination of the previous two inequalities reads
 \[
     \inf_{x \in S_1} \norm{Ax}_2 \ge \frac{\l_2}{2 K \sqrt{n}} \cdot |A_{11}-Y^T (B^T)^{-1}X|
 \]
  Recall that according to \eqref{eq: A precise}, $A_{11}= \sqrt{-1}\e s+d$, where $s$ is a real $N(0,1)$ random variable, and $d$ is independent of $s$. Conditioning on everything but $s$, we can treat $d$ and $Y^T (B^T)^{-1}X$ as constants. An elementary estimate using the normal density yields
  \[
    \P_s(|A_{11}-Y^T (B^T)^{-1}X| \le \mu) \le C \frac{\mu}{\e} \quad \text{for all } \mu >0.
  \]
  Applying this estimate with $\mu=\frac{2 K \sqrt{n}}{\l_2} \cdot t$ and integrating over the other random variables, we obtain
  \[
   \P ( \inf_{x \in S_1} \norm{Ax}_2 \le t)
   \le C \frac{2 K \sqrt{n}}{\l_2 \e} \cdot t \le C' \sqrt{n} \cdot t^c
  \]
  for some $c>0$ if $\l_2$ is chosen appropriately.
     \vskip 0.05in
   {\small \sffamily
   {\bf Choice of the parameters}.
   The inequality
   \[
    \frac{1}{\l_2 \e} \cdot t \le t^c, \ c>0
   \]
   holds with $c=0.2$ if we set $\l_2= t^{0.2}$.
   The constraint
   \[
     \e K_0 \sqrt{n} \le K,
   \]
   appearing above, is satisfied since we have chosen $\e=t^{0.6}$ and $K_0=t^{-0.05}$.

   One can try to tweak the parameters $\l_1, \l_2$, and $\e$ to cover all possible scenarios. This attempt, however, is doomed to fail since the system of the constraints becomes inconsistent. Indeed, to include all matrices $B$ in Cases 1 and 2, we have to choose $\l_2 \le \l_1 \e$. With this choice,
   \[
     \frac{t}{\l_2 \e} \ge  \frac{t}{\l_1 \e^2} >1,
   \]
   because of the constraint $K_0^2 n \e^2 \le t/2$. This forces us to consider the intermediate case.
   }
   \vskip 0.05in

  \subsection{Case 3, intermediate: $B$ is invertible, but not nicely invertible.}
  Assume that $\l_1 \e \le \smin(B) \le \l_2$ with $\l_2, \l_1$ defined in Cases 1 and 2. This is the most delicate case. Here we will have to rely on both local and global perturbations.
  We proceed like in Case 2 by multiplying $Ax$ from the left by a vector which  eliminates the dependence on all coordinates of $x$, except the first one.
  To this end,  note that
  \[
    \begin{bmatrix}
      1 & -Y^T (B^T)^{-1}
    \end{bmatrix}
    \cdot
    \begin{bmatrix}
      Y^T \\ B^T
    \end{bmatrix}
    = 0.
  \]
 Hence, for any $x \in S_1$,
  \begin{align*}
    \norm{A x}_2
    &\ge \frac{1}{\norm{\begin{bmatrix}      1 & -Y^T (B^T)^{-1}    \end{bmatrix}}_2}
      \left|
      \begin{bmatrix}
      1 & -Y^T (B^T)^{-1}
    \end{bmatrix}
    \cdot
    \begin{bmatrix}
      A_{11} &  Y^T \\ X & B^T
    \end{bmatrix}
    \cdot
      \begin{bmatrix}
      x_1 \\ \tilde{x}
    \end{bmatrix}
    \right| \\
    &\ge \frac{1}{1+\norm{Y^T (B^T)^{-1}}_2}
     \left| (A_{11}- Y^T (B^T)^{-1}X)x_1 \right| \\
    &\ge  \frac{1}{1+\norm{Y^T (B^T)^{-1}}_2}  |A_{11}-Y^T (B^T)^{-1}X| \cdot \frac{1}{\sqrt{n}}.
  \end{align*}
  Since the right hand side is independent of $x$, we can take the infimum over $x \in S_1$.

   Note that $Y^T (B^T)^{-1}$ is independent of $s$, see \eqref{eq: A precise}.
   We consider two subcases.
  If $\norm{Y^T (B^T)^{-1}}_2 \le \l_2^{-1}$,  then
  \[
   \inf_{x \in S_1} \norm{A x}_2 \ge  \frac{\l_2}{2  \sqrt{n}}  |A_{11}-Y^T (B^T)^{-1}X|,
  \]
   and we can finish the proof exactly like in Case 2, by conditioning on everything except $s$, and estimating the probability with respect to $s$.

    The second subcase requires more work. Assume that
  $\norm{Y^T (B^T)^{-1}}_2 \ge \l_2^{-1}$. Then the  inequality above yields
  \[
   \inf_{x \in S_1} \norm{A x}_2 \ge  \frac{1}{2 \sqrt{n}\norm{Y^T (B^T)^{-1}}_2}  |A_{11}-Y^T (B^T)^{-1}X|.
  \]
  Since we do not have a satisfactory upper bound for $\norm{Y^T (B^T)^{-1}}_2$, we cannot rely on $A_{11}$ to estimate the small ball probability. The second term in the numerator looks more promising, because it contains the same vector $Y^T (B^T)^{-1}$. This term, however, is difficult to analyze, since the random vectors $X$ and $Y$ are dependent. A simplification of both numerator and denominator would allow us to get rid of this dependence.

  We start with analyzing the denominator. By \eqref{eq: A precise}, $Y=rv-\e Z$, so
  \[
    \norm{Y^T (B^T)^{-1}}_2 \le \norm{v^T (B^T)^{-1}}_2+ \e \norm{Z^T (B^T)^{-1}}_2.
  \]
  As in the previous cases, disregarding an event of a small probability, we can assume that $\norm{Z}_2 \le K_0 \sqrt{n}$. Then by the assumption on $\smin(B)$,
  \begin{equation*}
    \e \norm{Z^T (B^T)^{-1}}_2 \le \frac{\e K_0 \sqrt{n}}{\smin(B)}
    \le \frac{ K_0 \sqrt{n}}{\l_1}.
  \end{equation*}
  The parameters $K_0, \l_1$, and $\l_2$ can be chosen so that $\frac{ K_0 \sqrt{n}}{\l_1} \le \l_2^{-1}/2$. Then, since
  by assumption $\norm{Y^T (B^T)^{-1}}_2 \ge \l_2^{-1}$, we conclude that
 \[
    \norm{Y^T (B^T)^{-1}}_2 \le 2 \norm{v^T (B^T)^{-1}}_2
  \]
  and
  \[
   \inf_{x \in S_1} \norm{A x}_2 \ge  \frac{1}{4 \sqrt{n}\norm{v^T (B^T)^{-1}}_2} \cdot |A_{11}-Y^T (B^T)^{-1}X|.
  \]
  The denominator here is independent of our random parameters.

  Now we pass to the analysis of the numerator.
   From \eqref{eq: A precise} it follows that $A_{11}-Y^T (B^T)^{-1}X=\a r+\b$ is a linear function of $r$
   with coefficients $\a$ and $\b$, which depend on other random parameters. This representation would allow us to filter out several complicated terms in $A_{11}-Y^T (B^T)^{-1}X$ by using the global perturbation $r$.

  Let $\l_3>0$ be a ``small'' parameter: $\l_3 = t^{a_3}$. Condition on everything except $r$. Since $r$ is uniformly distributed over the unit circle in $\C$, an easy density calculation yields
  \begin{equation}  \label{eq: on the crcle}
    \P_r (|\a r+b| \ge \l_3 |\a|) \ge 1-C \l_3.
  \end{equation}
  Taking the expectation with respect to the other random variables shows that the same bound holds unconditionally. Thus, disregarding the event of a small probability $C \l_3$, we obtain that $|A_{11}-Y^T (B^T)^{-1}X| \ge \l_3 |\a|$. The coefficient $\a$ in turn can be represented as follows: $\a=\a'- \e v^T (B^T)^{-1}Z$, where $\a' \in \C$ is independent of $Z$. Incorporating this into the bound above, we obtain
  \[
   \inf_{x \in S_1} \norm{A x}_2 \ge  \frac{\l_3}{4 \sqrt{n}\norm{v^T (B^T)^{-1}}_2}  |\a'- \e v^T (B^T)^{-1}Z|.
  \]
  Using the global perturbation allowed us to simplify the numerator and expose its dependence on the local perturbation variable $Z$. We will finish the proof using this local perturbation.

  Set $h^T=v^T (B^T)^{-1}/\norm{v^T (B^T)^{-1}}_2$ and recall that $h \in \C^{n-1}$ is independent of $Z$. Conditioning on everything except $Z$, we see that
  \[
   g:= \frac{\a'}{\norm{v^T (B^T)^{-1}}_2}- \e h^T Z= \text{const} + \e \g',
  \]
 where $\g'$ is a complex normal random variable of unit variance: $\E|\g'|^2=1$. Hence, as before, for any $\mu>0$
  \[
   \P_Z(|g| \le \mu) \le C \mu/\e,
  \]
  and integrating over other random variables,  we conclude that the same estimate holds unconditionally. Combining  this inequality with the previous one
  and recalling that we dropped an event of probability $C \l_3$ while using \eqref{eq: on the crcle}, we obtain
  \[
   \P (\inf_{x \in S_1} \norm{A x}_2 \le t)
   \le \P \left(|g| \le \frac{4 \sqrt{n}}{\l_3} t \right) +C \l_3
   \le C \frac{4 \sqrt{n}}{\l_3 \e} t  +C \l_3 \le C' \sqrt{n} t^{c'}
  \]
  for some $c'>0$.
  Choosing appropriate constants $a$ and $a_3$ in $\e=t^a$ and $\l_3=t^{a_3}$ finishes the proof in this case and completes the proof of Theorem \ref{th: unitary}.

     \vskip 0.05in
   {\small \sffamily
   {\bf Choice of the parameters}.
   The analysis of this case requires the following two constraints:
   \[
      \frac{K_0 \sqrt{n}}{\l_1} \le \frac{\l_2^{-1}}{2} \quad \text{and} \quad \frac{t}{\l_3 \e}+\l_3 \le t^{c'}, \ c>0.
   \]
   The first one is satisfied with the choice $K_0=t^{-0.05}, \ \l_1=t^{0.1}, \l_2=t^{0.2}$ that we made above.
   To satisfy the second one, set $\l_3= t^{0.2}$.
   }
\end{proof}
We made no effort to optimize the dependence on $t$ and $n$ in the proof above. It would be interesting to find the optimal bound here. Another interesting question, suggested by Djalil Chafai, is to analyze the behavior of the smallest singular value of the matrix $D+U$ where $U$ is uniformly distributed over a discrete subgroup of the unitary group. The case of the permutation group may be of special interest, because of its relevance for random graph theory. This question may require a combination of tools from Sections \ref{sec: continuous}--\ref{sec: unitary}, since both obstacles, the arithmetic structure and the lack of independence, make an appearance here.

\newpage
 \bibliographystyle{amsplain}

\end{document}